\documentclass[10pt, a4paper]{article} 
\author{Rafael Granero-Belinch\'on
\\{\footnotesize Departamento de Matem\'aticas, Estad\'istica y Computaci\'on}
\\{\footnotesize Universidad de Cantabria}
\\{\footnotesize Santander, Espa\~na}
\\{\footnotesize email: {\it rafael.granero@unican.es}}
\and 
Stefano Scrobogna
\\{\footnotesize Basque Center for Applied Mathematics}
\\{\footnotesize Alameda de Mazarredo, 14}
\\{\footnotesize Bilbao, Espa\~na}
\\{\footnotesize email: {\it sscrobogna@bcamath.org}}
}
\usepackage[english]{babel}
\usepackage{pst-grad} 
\usepackage{pst-plot} 
\usepackage{pstricks}
\usepackage{amsmath,amssymb,mathrsfs,amsthm, mathtools}
\usepackage{tikz} 
\usepackage{mathcomp,wasysym}  
\usepackage{graphicx}  
\usepackage[all]{xy} \xyoption{arc} \xyoption{color}
\usepackage{epsfig}
\usepackage{cite}
\usepackage[a4paper,
left=2.5cm, right=2.5cm,
top=3cm, bottom=3cm]{geometry}
\usepackage[a4paper]{geometry}
\usepackage{empheq}
\usepackage{fourier}
\DeclareMathAlphabet{\mathcal}{OMS}{cmsy}{m}{n}

\newcommand{\dx}{\textnormal{d}{x}}
\newcommand{\dt}{\textnormal{d}{s}}
\newcommand{\ddt}{\frac{\textnormal{d}}{ \textnormal{d}{t}}}

\newcommand{\sgn}{\textnormal{sgn}}
\newcommand{\pare}[1]{\left( #1 \right)}
\newcommand{\norm}[1]{\left\| #1 \right\|}
\newcommand{\av}[1]{\left| #1 \right|}
\newcommand{\bra}[1]{\left[ #1 \right]}
\newcommand{\set}[1]{\left\{ #1 \right\}}

\newcommand{\sumf}{\sum_{\left| q-q' \right|\leqslant4}}
\newcommand{\sumi}{\sum_{q'>q-4}}
\newcommand{\tqS}{\triangle_q}
\newcommand{\Sq}{S_q}
\newcommand{\tQS}{\triangle_{q'-1}}
\newcommand{\SQ}{S_{q'-1}}
\newcommand{\tv}{\triangle_q}
\newcommand{\Tv}{\triangle_{q'}}

\newcommand{\Hs}{H^s}

\newcommand{\cC}{\mathcal{C}}

\newcommand{\cF}{\mathcal{F}}
\newcommand{\cJ}{\mathcal{J}}
\newcommand{\cD}{\mathcal{D}}

\newcommand{\bR}{\mathbb{R}}
\newcommand{\bT}{\mathbb{T}}
\newcommand{\bZ}{\mathbb{Z}}

\newcommand{\cH}{\mathcal{H}}

\newcommand{\cN}{\mathcal{N}}
\newcommand{\cM}{\mathcal{M}}

\newcommand{\pat}{\partial_t}

\normalsize
\normalsize
\usepackage[bookmarks,colorlinks]{hyperref}

\newcommand{\hra}{\hookrightarrow}

\def\comm#1#2{{\left\llbracket#1,#2\right\rrbracket}}

\theoremstyle{theorem}
\newtheorem{theorem}{Theorem}[section]
\newtheorem*{theorem*}{Theorem}
\newtheorem{prop}[theorem]{Proposition}
\newtheorem{lemma}[theorem]{Lemma}

\theoremstyle{definition}

\newtheorem{rem}[theorem]{Remark}

\numberwithin{equation}{section}

\title{On an asymptotic model for free boundary Darcy flow in porous media}

\begin{document}
\maketitle

\begin{abstract}
We provide a rigorous mathematical study of an asymptotic model describing Darcy flow with free boundary in a small amplitude/large wavelength approximation. In particular, we prove several well-posedness results in critical spaces. Furthermore, we also study how the solution decays towards the flat equilibrium.
\end{abstract}

{\small
\tableofcontents}



\allowdisplaybreaks
\section{Introduction}
Since the pioneer works of Boussinesq \cite{Bo1872}, the derivation and study of asymptotic models for free boundary flows (usually, the water waves problem) are a hot research area (the interested reader can refer to \cite{lannes2013water}).

In this paper we study an asymptotic model for the intrusion of water into oil sand. This is known as the Muskat problem \cite{Musk}. The study of the (full) Muskat problem has received a lot of attention in the last years, and, as a consequence, there is a large literature available (we refer to \cite{GL} for a recent survey of the available results). More precisely, this work considers the following one-dimensional equation
\begin{equation}\label{eq:AD}\tag{AD$ _\nu $}
\left\lbrace
\begin{aligned}
&
\pat f=-\nu \Lambda^3 f-\Lambda f+\partial_{x} \left(\comm{\mathcal{H}}{f}\left(\nu \Lambda^3 f+\Lambda f\right)\right),
\\
& \left. f \right|_{t=0} = f_0
\end{aligned}
\right. 
\end{equation}
where $\comm{A}{B}=AB-BA$ denotes the commutator between $A$ and $B$, the Bond number $ \nu \geq 0 $ represents the ratio between capillarity and gravitational effects. and the Hilbert transform $\mathcal{H}$ and the Calderon operator $ \Lambda $ are defined as the following Fourier multiplier operators
\begin{align}\label{Hilbert}
\widehat{\mathcal{H}f}(n)=-i\text{sgn}(n) \hat{f}(n) , &&
\widehat{\Lambda f}(n)&=|n| \ \hat{f}(n)\,.
\end{align}
Equation \eqref{eq:AD} was derived by the authors as an asymptotic model for the Darcy flow in porous media under the assumption that the amplitude over the wavelength, a quotient known as \emph{steepness}, is small \cite{GrScr} (see also \cite{CGSW, GSdww, GS2}). In the gravity driven case (when $ \nu =0 $) the equation \eqref{eq:AD} reads as
\begin{equation}\label{eq:AD0}\tag{AD$ _0 $}
\left\lbrace
\begin{aligned}
&\partial_t f +\Lambda f = \Lambda\pare{f\ \Lambda f} + \partial_x \pare{f\ \partial_x f}, \\
& \left.f\right|_{t=0} = f_0.
\end{aligned}
\right. 
\end{equation}
Although \eqref{eq:AD} and \eqref{eq:AD0} seem semilinear fourth and second order nonlocal PDE respectively (due to terms like $f\partial_x^4 f$ and $f\partial_x^2 f$), the commutator structure of the nonlinearity implies that they are a quasilinear third and a fully nonlinear first order nonlocal PDE respectively (see Lemma \ref{lemma1} below).

There are several motivations to study asymptotic models of free boundary Darcy flow. One of them arises from computational reasons. The idea is then, to simulate the asymptotic model to obtain an accurate description of the full problem at a lower computational cost. In that regards, let us briefly emphasize that the one-phase Muskat problem reveals itself as somehow harder (computationally speaking) than the two-phase Muskat problem. Let us try to briefly explain the reasons for this ``\emph{paradoxical}" fact. In the case when there are two fluids, the gravity driven Muskat problem reads as the following single nonlocal pde \cite{c-g07}
$$
\partial_t f=\text{p.v.}\int_\bR \frac{ \pare{ \partial_x f(x)-\partial_x f(x-y)}y}{ \pare{ f(x)-f(x-y)}^2+y^2}dy.
$$
Simplified models for this case where provided (following heuristic ideas) by C\'ordoba, Gancedo \& Orive \cite{c-g-o08} (see also \cite{GNO}). Remarkably, when the one phase Muskat problem is considered, the previous pde has to be modified and one is forced to study the following system of a nonlocal pde and an integral equation \cite{c-c-g10}
\begin{align*}
\partial_t f(x)&=\text{p.v.}\int_\bR \varpi(x-y)\frac{y}{y^2+(f(x)-f(x-y))^2}dy-\partial_x f(x)\text{p.v.}\int_\bR \varpi(x-y)\frac{f(x)-f(x-y)}{y^2+(f(x)-f(x-y))^2}dy,\\
-\partial_x f(x)&=\text{p.v.}\int_\bR \varpi(\beta) \mathcal{B} (x,f(x),\beta,f(\beta))d\beta\cdot (1,\partial_x f (x))+\frac{\varpi(x)}{2},
\end{align*}
where $\mathcal{B} $ denotes the kernel of $\nabla^\perp\Delta^{-1}$, \emph{i.e.}
$$
\mathcal{B} (x_1,x_2,y_1,y_2)=\left(-\frac{x_2-y_2}{(x_2-y_2)^2+(x_1-y_1)^2}, \frac{x_1-y_1}{(x_2-y_2)^2+(x_1-y_1)^2}\right).
$$
Thus, to write the amplitude of the vorticity in terms of the interface, one needs to invert an operator as in C{\'o}rdoba,  C{\'o}rdoba, \& Gancedo \cite{c-c-g10}. This is, mathematically and computationally, a challenging issue.

Another reason is the possibility of finding new finite time singularity scenarios. In particular, for the two-phase Muskat problem, Castro, C\'ordoba, Fefferman, Gancedo \& L\'opez-Fern\'andez \cite{ccfgl} proved the existence of \emph{turning waves}, \emph{i.e.} interfaces that can be parametrized as a smooth graphs at time $t=0$ but that become smooth curves that cannot be parametrized as graphs after a finite time (see also \cite{CGO, BCG, GG, cordoba2015note, cordoba2017note}). We observe that these turning waves are interfaces such that there exists $0<T_{1}<\infty$ and
$$
\limsup_{t\rightarrow T_{1}} \|\partial_x f(t)\|_{L^\infty}=\infty.
$$ 
In the case of the one-phase Muskat problem, Castro, C\'ordoba, Fefferman, Gancedo \cite{ccfgonephase} proved the existence of curves that self intersect in finite time in what is called a splash singularity (see also \cite{gancedo2014absence, coutand2016impossibility, FIL, cponephase}. It remains as an interesting open problem whether the Muskat problem can have a \emph{cusp} singularity, \emph{i.e.} a singularity where the slope and curvature of the interface blows up while the interface remains a graph (see \cite{ccgs-10, G, CGS, GS, CGSVfiniteslope, cameron2017global} for global existence results). In that regards, we are optimistic enough to think that such an scenario should be simpler to prove (or discard) in an asymptotic model rather than the full problem. Taking this into consideration, the results of this paper then give conditions that excludes finite time blow ups of turning type.

\subsection{Functional spaces}
The space domain considered in the present article is the one-dimensional torus. i.e. $\bT = \left. \bR \big. \right/ 2\pi \ \bZ. $ The domain $ \bT $ can also be understood as the interval $ \bra{-\pi, \pi} $ endowed with periodic boundary conditions. Let $f(x)$ denote a $L^2$ function on $\bT$. Then, its Fourier transform is an $ \ell^2\pare{\bZ} $ sequence defined as
$$
\hat{f}{(n)} = \hat{f}_n=\frac{1}{\sqrt{2\pi}}\int_\bT f(x)e^{-ix \ n}\dx,
$$
for any $ n\in \bZ $, 
with inverse Fourier transform
$$
f(x)=\frac{1}{\sqrt{2\pi}}\sum_{n\in \bZ} \hat{f}{(n)} \ e^{i  x \ n}.
$$
Then, we define the $L^2-$based (homogeneous) Sobolev spaces $\dot{H}^\alpha(\bT)$
\begin{equation*}\label{Sobhomo}
\dot{H}^\alpha(\bT)=\left\{u(x)\text{ such that } \|u\|_{\dot{H}^\alpha(\bT)}^2:=\sum_{k} \av{k}^{2\alpha}\av{\hat{u}{(k)}}^2<\infty
\right\},
\end{equation*}

In the present work we will need to define Sobolev spaces with fractional derivatives and integrability indexes different from two, i.e. spaces $ \dot{W}^{s, p} \pare{\bT} $ for $ s\in \bR $ and $ p\in \bra{1, \infty} $. We provide here a characterization of such spaces using using the theory of Littlewood-Paley (see Appendix \ref{elements LP}). If we consider the dyadic block $ \tqS $ we can hence define the semi-norm
\begin{equation*}
\norm{u}_{\dot{W}^{s, p}\pare{\bT}} = \pare{\sum_{q\in \bZ} 2^{pqs}\norm{\tqS u}_{L^p\pare{\bT}}^p}^{1/p}. 
\end{equation*}
{It is well-known that
\begin{equation*}
\norm{u}_{\dot{W}^{s, 2}\pare{\bT}} \sim \norm{u}_{\dot{H}^{s}\pare{\bT}}, 
\end{equation*}
in the sense that there exists a positive absolute constant $ C> 0 $ s.t. $ \frac{1}{C}\norm{u}_{\dot{H}^{s}\pare{\bT}} \leqslant \norm{u}_{\dot{W}^{s, 2}\pare{\bT}} \leqslant C \norm{u}_{\dot{H}^{s}\pare{\bT}} $.}

Let us recall the following embedding, valid for mean-free distributions on the one-dimensional torus:
\begin{equation*}
\dot{H}^{1/2}\pare{\bT}\hra L^p\pare{\bT}, \ \forall \ p\in\left[1, \infty \right), 
\end{equation*}
and in particular the following inequality holds true for any $ u\in \dot{H}^{1/2} \cap L^p $, 
\begin{equation*}
\norm{u}_{L^p}\leqslant C \sqrt{p} \norm{u}_{H^{1/2}}. 
\end{equation*}
 We also recall that $ \Hs\pare{\bT} $ for $ s>1/2 $ embeds continuously in $ L^\infty\pare{\bT} $ with bound
\begin{equation}\label{eq:Sobolev_emb_with_constant}
\norm{u}_{L^\infty} \leqslant \frac{C}{\eta} \norm{u}_{H^{\frac{1}{2} + \eta}}. 
\end{equation}

Similarly, we define the (homogeneous) Wiener spaces $\dot{A}^\alpha(\bT)$ and Wiener spaces with weight as
\begin{equation*}\label{Wienerhomo}
\dot{{A}}^\alpha(\bT)=\left\{u(x)\text{ such that } \|u\|_{\dot{A}^\alpha(\bT)}:=\sum_{k} \av{k}^{\alpha}\av{\hat{u}{(k)}}<\infty\right\}.
\end{equation*}
\begin{equation*}\label{WienerhomoW}
\dot{{A}}^\alpha_\nu(\bT)=\left\{u(x)\text{ such that } \|u\|_{\dot{A}_\nu^\alpha(\bT)}:=\sum_{k} \av{k}^{\alpha}e^{\nu |k|}\av{\hat{u}{(k)}}<\infty\right\},
\end{equation*}
respectively. Obviously, $\dot{A}^\alpha=\dot{A}^\alpha_0$ and the functions in $A_\nu^s$ are analytic in the complex strip
$$
\mathbb{S}_\nu=\left\{x+iy,\;|y|<\nu\right\}.
$$

We observe that the equation \eqref{eq:AD} conserves the average of a solution $ f $. Thus, without losing generality, we can assume this average to be zero. Thus, from this point onwards we identify the spaces $ \dot{H}^s\pare{\bT} $ and $ {H}^s\pare{\bT} $,  $ \dot{A}^s\pare{\bT} $ and $ {A}^s\pare{\bT} $ and $\dot{A}^s_\nu\pare{\bT} $ and $ {A}^s_\nu\pare{\bT} $. Then, the previously defined semi-norms when restricted to average-free functions are in fact genuine norms.

\subsection{A word on scaling}\label{sec:scaling}
We observe that, when $\nu=0$, equation \eqref{eq:AD} is left invariant by the scaling
\begin{equation}\label{eq:scaling}
f_\lambda (x,t)=\frac{1}{\lambda}f(\lambda x,\lambda t).
\end{equation}
Remarkably, this scaling is the same as in the full Muskat problem (see \cite{CL}). Then, we note that the following spaces are critical for this scaling
$$
L^\infty \pare{ 0,T;\dot{A}^1(\bT)},\;
L^\infty \pare{ 0,T;\dot{H}^{3/2}(\bT)},\;
L^\infty  \pare{0,T;\dot{W}^{1,\infty}(\bT)}.
$$

Let us now consider the equation \eqref{eq:AD} for $ \nu >0 $. In such setting \eqref{eq:AD} reads, when expanded, as
\begin{equation}
\partial_t f + \nu \Lambda^3 f + \Lambda f =  \nu \bra{ \Big. \Lambda\pare{f\ \Lambda^3 f} - \partial_x \pare{f\ \partial_x^3 f}} + \bra{ \Big. \Lambda\pare{f\ \Lambda f} + \partial_x \pare{f\ \partial_x f}}. \label{eq:AD_nu>0}
\end{equation}
There is no scale invariance satisfied by the equation \eqref{eq:AD_nu>0}. Despite this fact we can rewrite \eqref{eq:AD_nu>0} as
\begin{equation}\label{eq:AD_nu>0_1}
\frac{1}{2} \partial_t f + \nu \Lambda^3 f - \nu { \Lambda\pare{f\ \Lambda^3 f} +  \nu\partial_x \pare{f\ \partial_x^3 f}} = 
-\frac{1}{2} \partial_t f- \Lambda f +\Lambda\pare{f\ \Lambda f} + \partial_x \pare{f\ \partial_x f}
,
\end{equation}
and then study separately the homogeneity of the left- and right-hand-side of \eqref{eq:AD_nu>0_1}. Concerning the right-hand-side the deductions of above still hold and it invariant with respect to the scaling \eqref{eq:scaling}. The left-hand-side of \eqref{eq:AD_nu>0_1} is invariant with respect to the scaling 
\begin{equation}\label{eq:scaling_nu>0}
f_\lambda \pare{x, t} = \frac{1}{\lambda} f \pare{\lambda x, \lambda^3 t}.
\end{equation}
We remark that the space $ L^\infty_t\pare{\bR_+} $ is invariant respect both the dilations $ t\mapsto \lambda t $ and $ t\mapsto \lambda^3 t $ which characterize respectively \eqref{eq:scaling} and \eqref{eq:scaling_nu>0}. Then, the spaces 
\begin{equation*}
L^\infty\pare{0, T ; \dot{H}^{3/2}\pare{\bT}},\quad L^\infty\pare{0, T ; \dot{A}^{1}\pare{\bT}} 
\end{equation*} 
are, in a certain sense relatively to scale invariance, reasonable critical spaces to consider for the equation \eqref{eq:AD_nu>0}. 

\section{Main results \& Discussion}
\subsection{Results for the gravity driven case \eqref{eq:AD0}}

Our first result proves the existence of solution for analytic initial of arbitrary size.
\begin{theorem}\label{thm:analytic}
Let $f_0\in A^1_1(\mathbb{T})$ be the initial data for \eqref{eq:AD0}. Then there exist a short enough time $T$,
$$
0<T\leq \frac{1}{4\|f_0\|_{A^1_1}}
$$
and a unique mild solution to \eqref{eq:AD0} 
$$ 
f \in L^\infty \pare{0,T;A^1_1}\cap C \pare{[0,T],A^1_{0.5}}.
$$ 
\end{theorem}
This result is a sort of Cauchy-Kovalevsky Theorem. However, we emphasize that the solution is less regular in time (merely $L^\infty$ instead of continuous) but maintains its original strip of analyticity.

Furthermore, we are able to prove a result establishing the decay of certain norms of the solution:
\begin{prop}\label{thm:Decay}
Assume that the solution to \eqref{eq:AD0} satisfies 
$$
\sup_{0\leq t\leq T}\norm{f(t)}_{\dot{A}^1}\leqslant 1,
$$
then 
$$
\sup_{0\leq t\leq T}\max_{y}f(y,t)\leqslant \max_{y}f_0(y),
$$ 
$$
\inf_{0\leq t\leq T}\min_{y}f(y,t)\geq \min_{y}f_0(y),
$$
and, as a consequence,
$$
\sup_{0\leq t\leq T}\norm{f(t)}_{L^\infty}\leqslant \|f_0\|_{L^\infty}.
$$
\end{prop}

This result seems the analog of the maximum principles known for the full Muskat problem \cite{c-g09}. In particular, when compared to the results in \cite{c-g09,CGO}, it seems that the model equation \eqref{eq:AD0} is less stable than the full Muskat problem.

Once the local solution for analytic data is known, we turn our attention to the global solution for initial data satisfying certain size restrictions in critical spaces. Then, our well-posedness result for Wiener class initial data reads as follows
\begin{theorem}\label{thm:GWP_A1}
Let $f_0\in A^1$ be the initial data for \eqref{eq:AD0}. Assume that
\begin{equation*}
\norm{f_0}_{\dot{A}^1}< 1/2, 
\end{equation*}
then the Cauchy problem \eqref{eq:AD0} is globally well-posed and admits a unique solution 
$$ 
f \in L^\infty\pare{\bR_+; A^{1}\pare{\bT} }\cap \mathcal{M}\pare{\bR_+; A^{2}\pare{\bT} },
$$ 
verifying
$$
\norm{f(t)}_{A^1}\leqslant \norm{f_0}_{A^1},
$$
and
$$
\norm{f(t)}_{A^0}\leqslant \norm{f_0}_{A^0}e^{-\left(1-2\norm{f_0}_{A^1}\right)t}.
$$
\end{theorem}
We observe that this is a global existence and decay for initial data in a critical space.


The next result we prove in the present manuscript is a global existence result for initial data in $ H^{\frac{3}{2}} \cap H^{\frac{3}{2}+\varepsilon}, \ \varepsilon >0 $, we suppose the initial data to be small in $ H^{\frac{3}{2}}  $ and of arbitrary size in $H^{\frac{3}{2}+\varepsilon} $.
Very recently D. Cordoba  and O. Lazar proved in \cite{CL} that it possible to construct global solutions for the 2D (full) Muskat problem when the initial data is small in the space $ H^{3/2}$ and also belongs to $H^{5/2} $. 

The following theorem establishes a similar result for the evolution equation \eqref{eq:AD0};

\begin{theorem}\label{thm:GWP_H3/2}
Let us suppose $ f_0 \in H^{\frac{3}{2}} \cap H^{\frac{3}{2}+\varepsilon}, \ \varepsilon>0 $
There exists a constant $ C>0 $ such that, if $ f_0 $ is such that
\begin{equation*}
\norm{f_0}_{H^{\frac{3}{2}}} <
\min \set{1 \  , \frac{\varepsilon^2}{C \ \pare{\norm{f_0}_{H^{\frac{3}{2}+\varepsilon}} \log \norm{f_0}_{H^{\frac{3}{2}+\varepsilon}} +1}^2}  } , 
\end{equation*} 
the Cauchy problem \eqref{eq:AD0} is globally well-posed and admits a unique solution $ f $ such that for any $ T>0 $
\begin{equation*}
f\in C^{0, \frac{1}{2} -\vartheta} \pare{\left[0, T\right) ; H^{1+\vartheta\pare{\varepsilon+1}}} \cap L^\infty\pare{\left[0, T\right); H^{\frac{3}{2}+\varepsilon}} \cap L^2\pare{\left[0, T\right); H^{2 +\varepsilon}}, \hspace{1cm} \frac{1}{2\pare{\varepsilon+1}} \leqslant \vartheta <\frac{1}{2} . 
\end{equation*}
\end{theorem}
In other words, the initial data must be small in $H^{3/2}$ even for large $\epsilon$. Furthermore, if the initial data is large in $H^{3/2+\varepsilon}$, then the size condition for the initial data in $H^{3/2}$ is even more restrictive.

The methodology used in order to prove Theorem \ref{thm:GWP_H3/2} differs completely from the one used in \cite{CL}, since \eqref{eq:AD0} presents a nonlinearity of polynomial type we will be able to use tools characteristic of the paradifferentail calculus, we refer the reader to \cite[Chapter 2]{BCD}; we will hence decompose the nonlinearity in an infinite sum of elementary packets on which we will be able to highlight some nontrivial regularizing properties of the equation \eqref{eq:AD0}.

\subsection{Results for the gravity-capillary driven case \eqref{eq:AD}}


We are able to extend Theorem \ref{thm:GWP_A1} to the case with surface tension:
\begin{theorem}\label{thm:GWP_A1surf}
Let $\nu>0$ be a fixed parameter and $f_0\in A^1$ be the initial data for \eqref{eq:AD}. Assume that
\begin{equation*}
\norm{f_0}_{\dot{A}^1}< 1/2, 
\end{equation*}
then the Cauchy problem \eqref{eq:AD} is globally well-posed and admits a unique solution 
$$ 
f \in L^\infty\pare{\bR_+; A^{1}\pare{\bT} }\cap \mathcal{M}\pare{\bR_+; A^{4}\pare{\bT} },
$$ 
verifying
$$
\norm{f(t)}_{A^1}\leqslant \norm{f_0}_{A^1},
$$
and
$$
\norm{f(t)}_{A^0}\leqslant \norm{f_0}_{A^0}e^{-\left(1-2\norm{f_0}_{A^1}\right)(1+\nu)t}.
$$
\end{theorem}

\begin{rem} We observe that the size condition is $\nu-$independent.
\end{rem}

The following result is a result analogous of the one stated in Theorem \ref{thm:GWP_H3/2} for the system \eqref{eq:AD} when $ \nu > 0 $: 

\begin{theorem}\label{thm:GWP_H2}
Let us suppose that $ f_0 $ is zero mean function on $ \bT $, $ f_0\in H^2\pare{\bT} $ such that 
\begin{equation}\label{eq:smallness_hyp_H2}
\norm{f_0}_{H^2\pare{\bT}} \leqslant \frac{1}{C} \ \min \set{1, \ \nu^{-\frac{1}{4}}}. 
\end{equation}
Then there exists a unique solution $ f \in \cC\pare{\bR_+; H^2}\cap L^2\pare{\bR_+; H^{\frac{7}{2}}} $ of the Cauchy problem \eqref{eq:AD} such that for any $ t> 0 $
\begin{equation*}
\norm{f\pare{t}}_{H^2\pare{\bT}}^2 + \int_0^t \bra{ \nu \norm{ \Lambda^{3/2} f\pare{s}}_{H^2\pare{\bT}}^2 + \norm{ \Lambda^{1/2} f\pare{s}}_{H^2\pare{\bT}}^2} \dt \leqslant \norm{f_0}_{H^2\pare{\bT}}^2.
\end{equation*}
\end{theorem}

\subsection{Discussion}
In this paper we prove several well-posedness results for an asymptotic model of free boundary Darcy flow. Most of them are global existence results in scale invariant spaces. In that regards, our results should be understood as nonlinear stability results rather than linear stability (even if some size conditions are imposed on the initial data). These results excludes the existence of turning waves singularities but leave open the door to cusp singularities. The occurrence of such behavior will be the object of future research.

\section{Gravity driven system \eqref{eq:AD0}}\label{sec:gravity}

\subsection{Proof of Theorem \ref{thm:analytic}}
We look for a solution of the form
$$
f(x,t)=\lambda \sum_{\ell=0}^\infty \lambda^\ell f^{(\ell)}(x,t),
$$
where $\lambda=\lambda (f_0)$ will be chosen below. Then, the existence of solution is reduced to the summability of a series where each term satisfy a linear problem. A similar idea can be tracked back to the works of Oseen \cite{Os1912} (see also \cite{CGSW}). Indeed, matching the appropriate terms, we find that $f^{(\ell)}$ satisfies
$$
\partial_tf^{(\ell)} +\Lambda f^{(\ell)}=\sum_{j=0}^{\ell-1} \Lambda\left(f^{(j)}\Lambda f^{(\ell-j-1)}\right)+\partial_x \left(f^{(j)}\partial_x f^{(\ell-j-1)}\right),
$$
with initial data
$$
f^{(\ell)}(x,0)=0 \text{ if }\ell\neq0
$$
and
$$
f^{(0)}(x,0)=\frac{f_0}{\lambda} \text{ otherwise.}
$$
Then
$$
f^{(0)}= e^{-t\Lambda}\frac{f_0}{\lambda},
$$
and we can solve recursively for the other $f^{(\ell)}$. In particular, using that the solution of
$$
\partial_t u(x,t)+\Lambda u(x,t)=F(x,t),\quad u(x,t)=g(x)
$$
is given by
$$
\hat{u}(k,t)=e^{-t|k|}\hat{g}(k)+\int_{0}^t e^{-(t-s)k}\hat{F}(k,s)ds,
$$
we find that
$$
\widehat{f^{(\ell)}}(k,t)=\sum_{j=0}^{\ell-1}\sum_{n=-\infty}^\infty\int_0^te^{-(t-s)|k|}\left[|k|\left( \widehat{f^{(j)}}(n,s)|k-n|\widehat{f^{(\ell-j-1)}}(k-n,s)\right)+ik\left( \widehat{f^{(j)}}(n,s)i(k-n)\widehat{f^{(\ell-j-1)}}(k-n,s)\right)\right].
$$
Thus, using
$$
|k||k-n|-k(k-n)\leq |k-n|^2+|n||k-n|-(k-n)^2-n(k-n),
$$
we have that
\begin{align}
\widehat{\partial_{x} \left(\comm{\mathcal{H}}{a}\Lambda b\right)}&=\hat{a}(k-n)\hat{b}(n)\left(|k||k-n|-k(k-n)\right)\nonumber\\
&\leqslant 2\left|\hat{a}(k-n)\hat{b}(n)|n||k-n|\right|.\label{commFourier}
\end{align}
Thus, we obtain that
$$
|\widehat{f^{(\ell)}}(k,t)|\leq 2\sum_{j=0}^{\ell-1}\sum_{n=-\infty}^\infty\int_0^te^{-(t-s)|k|}| \widehat{f^{(j)}}(n,s)||n||k-n||\widehat{f^{(\ell-j-1)}}(k-n,s)|ds.
$$
We now fix $1\ll R$ and consider the partial sum
$$
f_R(x,t)=\lambda \sum_{\ell=0}^R \lambda^\ell f^{(\ell)}(x,t).
$$
Then, $f_R$ satisfies
\begin{align*}
\partial_t f_R +\Lambda f_R&=\sum_{\ell=0}^R \lambda^{\ell+1}\sum_{j=0}^{\ell-1} \Lambda\left(f^{(j)}\Lambda f^{(\ell-j-1)}\right)+\partial_x \left(f^{(j)}\partial_x f^{(\ell-j-1)}\right)\\
&=\sum_{\ell=1}^R\sum_{j=0}^{\ell-1} \Lambda\left(\lambda^{j+1}f^{(j)}\lambda^{\ell-j-1+1}\Lambda f^{(\ell-j-1)}\right)+\partial_x \left(\lambda^{j+1}f^{(j)}\lambda^{\ell-j-1+1}\partial_x f^{(\ell-j-1)}\right),
\end{align*}
We estimate
\begin{equation}\label{estl=0}
\lambda\|f^{(0)}\|_{A^1_{1}}\leq \|f_0\|_{A^1_1}.
\end{equation}
Define
$$
\nu(\ell)=R-\ell+1.
$$
Applying Tonelli's theorem, we find that
\begin{align*}
\|f^{(\ell)}\|_{A^1_{\nu(\ell)}}&\leq 2\sum_{j=0}^{\ell-1}\sum_{k=-\infty}^\infty\sum_{n=-\infty}^\infty\int_0^te^{-(t-s)|k|} |k| e^{\nu(\ell)|k|}| \widehat{f^{(j)}}(n,s)||n||k-n||\widehat{f^{(\ell-j-1)}}(k-n,s)|ds\\
&\leq 2\sum_{j=0}^{\ell-1}\sum_{k=-\infty}^\infty\sum_{n=-\infty}^\infty\int_0^t e^{(\nu(\ell)+1)|k|}| \widehat{f^{(j)}}(n,s)||n||k-n||\widehat{f^{(\ell-j-1)}}(k-n,s)|ds
\end{align*}
where we have used 
$$
|k|\leq e^{|k|}.
$$
Now we observe that, since $0\leq j \leq \ell-1$ and $0\leq \ell-j-1 \leq \ell-1$ 
$$
\nu(\ell)+1=R-(\ell-1)+1\leq R-j+1=\nu(j),
$$
and
$$
\nu(\ell)+1=R-(\ell-1)+1\leq R-(\ell-j-1)+1=\nu(\ell-j-1).
$$
Thus,
\begin{align*}
\|f^{(\ell)}(t)\|_{A^1_{\nu(\ell)}}&\leq 2\sum_{j=0}^{\ell-1}\sum_{k=-\infty}^\infty\sum_{n=-\infty}^\infty\int_0^t e^{\nu(j)|n|}e^{\nu(\ell-j-1)|k-n|}| \widehat{f^{(j)}}(n,s)||n||k-n||\widehat{f^{(\ell-j-1)}}(k-n,s)|ds\\
&\leq 2\sum_{j=0}^{\ell-1}\int_0^t \|f^{(j)}(s)\|_{A^{1}_{\nu(j)}}\|f^{(\ell-j-1)}(s)\|_{A^1_{\nu(\ell-j-1)}}ds.
\end{align*}
We define the numbers
$$
\mathscr{A}_\ell=2\|f^{(\ell)}(t)\|_{A^1_{\nu(\ell)}},\quad \mathscr{A}_0=1.
$$
These numbers satisfy
\begin{align*}
\mathscr{A}_\ell&\leq \sum_{j=0}^{\ell-1}\int_0^t \mathscr{A}_j \mathscr{A}_{\ell-j-1}ds.
\end{align*}
Then, we prove by induction that
$$
\mathscr{A}_\ell\leq \mathscr{C}_\ell t^\ell,
$$
where $\mathscr{C}_\ell$ are the Catalan numbers. It is known that the Catalan numbers grow like $4^\ell$. We also observe that
$$
\nu(\ell)\geq 1
$$
for $\ell\leq R$. As a consequence, we have the bound
\begin{align*}
\|f_R(t)\|_{A^1_1}&\leq \lambda \|f^{(0)}\|_{A^1_1} +\lambda \sum_{\ell=1}^R \lambda^\ell \|f^{(\ell)}(t)\|_{A^{1}_1}\\
&\leq \lambda \|f^{(0)}\|_{A^1_1}+\lambda \sum_{\ell=1}^R \lambda^\ell \|f^{(\ell)}(t)\|_{A^{1}_{\nu(\ell)}}\\
&\leq \lambda \|f^{(0)}\|_{A^1_1}+\frac{\lambda}{2} \sum_{\ell=1}^R \lambda^\ell \mathscr{A}_\ell\\
&\leq \lambda \|f^{(0)}\|_{A^1_1}+\frac{\lambda}{2} \sum_{\ell=1}^R \lambda^\ell\mathscr{C}_\ell t^\ell\\
&\leq \|f_0\|_{A^1_1}+\frac{\lambda}{2} \sum_{\ell=1}^\infty (4 \lambda t)^\ell,
\end{align*}
where we have used \eqref{estl=0}. We choose $\lambda=\|f_0\|_{A^1_1}$ and note that, for $t< 1/4\|f_0\|_{A^1_1}$, we obtain the bound
\begin{align*}
\|f_R(t)\|_{A^1_1}&\leq \|f_0\|_{A^1_1}+\frac{\|f_0\|_{A^1_1}}{2} \frac{4 \|f_0\|_{A^1_1} t}{1- 4 \|f_0\|_{A^1_1} t}.
\end{align*}
Thus, we have a bound
$$
f_R\in L^\infty(0,T;A^{1}_1).
$$
This bound is independent of $R$ and then, using Weierstrass $M$-theorem in the space $L^\infty(0,T;A^{1}_1)$, we can pass to the limit as $R\rightarrow\infty$ and obtain 
$$
\lim_{R\rightarrow\infty}f_R=f\in L^\infty(0,T;A^{1}_1).
$$
Furthermore, $f_R$ satisfies
\begin{align*}
\widehat{f_R}(k,t)=e^{-t|k|}\hat{f_0}+\sum_{\ell=1}^R\sum_{j=0}^{\ell-1}\sum_{n=-\infty}^\infty\int_0^te^{-(t-s)|k|}\left(\lambda^{j+1}\widehat{f^{(j)}}(n,s)\lambda^{\ell-j-1+1}\widehat{f^{(\ell-j-1)}}(k-n,s)\right)\left[|k||k-n|-k(k-n)\right].
\end{align*}
Then, using the Cauchy theorem for the product of series, the limit is a mild solution of our problem
\begin{align*}
\widehat{f}(k,t)=e^{-t|k|}\hat{f_0}+\sum_{n=-\infty}^\infty\int_0^te^{-(t-s)|k|}\left(\widehat{f}(n,s)\widehat{f}(k-n,s)\right)\left[|k||k-n|-k(k-n)\right].
\end{align*}

\subsection{Proof of Proposition \ref{thm:Decay}}
We start this section with an auxiliary Lemma that establishes an integral formula for the nonlinear term
\begin{lemma}\label{lemma1}The nonlinear term $\partial_x \left(\comm{\mathcal{H}}{u}\Lambda u\right)$ admits the representation
\begin{align*}
\partial_x \left([\mathcal{H},u]\Lambda u\right)&=\frac{1}{16\pi^2}\textnormal{p.v.}\int_\mathbb{T}\textnormal{p.v.}\int_\mathbb{T}\frac{u(x)-u(x-y)}{\sin^2(y/2)}\frac{u(x-y)-u(x-y-z)}{\sin^2(z/2)} \textnormal{d} z \textnormal{d} y+(\partial_x u(x))^2\\
&=\frac{1}{4\pi}\textnormal{p.v.}\int_\mathbb{T}\frac{u(x)-u(x-y)}{\sin^2(y/2)}\Lambda u(x-y)\textnormal{d} y+(\partial_x u(x))^2
\end{align*}
\end{lemma}
\begin{proof}[Proof of Lemma \ref{lemma1}]
First, we observe that
$$
\comm{\mathcal{H}}{f}g=\frac{1}{2\pi}\textnormal{p.v.}\int_\mathbb{T}\frac{ \pare{ f(x-y)-f(x)}g(x-y)}{\tan(y/2)}\textnormal{d} y,
$$
\begin{equation}\label{auxiliar}
\mathcal{H}\left(\mathcal{H}g\right)(x)=\frac{1}{2\pi}\textnormal{p.v.}\int_\mathbb{T}\frac{\cH g(x-y)}{\tan(y/2)}\textnormal{d} y=\frac{1}{4\pi^2}\textnormal{p.v.}\int_\mathbb{T}\textnormal{p.v.}\int_\mathbb{T}\frac{g(x-y-z)}{\tan(z/2)}\frac{1}{\tan(y/2)}\textnormal{d} z \ \textnormal{d} y.
\end{equation}
Thus, we obtain that
\begin{align*}
\partial_x \left(\comm{\mathcal{H}}{u}\Lambda u\right)&=\frac{\partial_x}{2\pi}\textnormal{p.v.}\int_\mathbb{T}\frac{u(x-y)-u(x)}{\tan(y/2)}\mathcal{H} \partial_x u(x-y)\textnormal{d} y\\
&=\frac{\partial_x}{4\pi^2}\textnormal{p.v.}\int_\mathbb{T}\frac{u(x-y)-u(x)}{\tan(y/2)}\textnormal{p.v.}\int_{\mathbb{T}} \frac{\partial_x u(x-y-z)}{\tan(z/2)} \textnormal{d} z \textnormal{d} y\\
&=\frac{\partial_x}{4\pi^2}\textnormal{p.v.}\int_\mathbb{T}\textnormal{p.v.}\int_\mathbb{T}\frac{u(x-y)-u(x)}{\tan(y/2)}\frac{\partial_x u(x-y-z)}{\tan(z/2)} \textnormal{d} z \textnormal{d} y\\
&= I_1+I_2+I_3,
\end{align*}
where
$$
I_1=\frac{1}{4\pi^2}\textnormal{p.v.}\int_\mathbb{T}\textnormal{p.v.}\int_\mathbb{T}\frac{\partial_x u(x-y)}{\tan(y/2)}\frac{\partial_x u(x-y-z)}{\tan(z/2)} \textnormal{d} z \textnormal{d} y,
$$
$$
I_2=-\frac{\partial_x u(x)}{4\pi^2}\textnormal{p.v.}\int_\mathbb{T}\textnormal{p.v.}\int_\mathbb{T}\frac{1}{\tan(y/2)}\frac{\partial_x u(x-y-z)}{\tan(z/2)} \textnormal{d} z \textnormal{d} y,
$$
$$
I_3=\frac{1}{4\pi^2}\textnormal{p.v.}\int_\mathbb{T}\textnormal{p.v.}\int_\mathbb{T}\frac{u(x-y)-u(x)}{\tan(y/2)}\frac{\partial_x^2 u(x-y-z)}{\tan(z/2)} \textnormal{d} z \textnormal{d} y.
$$
We have that
\begin{align*}
I_1&=\frac{1}{4\pi^2}\textnormal{p.v.}\int_\mathbb{T}\textnormal{p.v.}\int_\mathbb{T}\frac{-\partial_y u(x-y)}{\tan(y/2)}\frac{\partial_x u(x-y-z)}{\tan(z/2)} \textnormal{d} z \textnormal{d} y\\
&=\frac{1}{4\pi^2}\textnormal{p.v.}\int_\mathbb{T}\textnormal{p.v.}\int_\mathbb{T}\frac{\partial_y\left(u(x)-u(x-y)\right)}{\tan(y/2)}\frac{\partial_x u(x-y-z)}{\tan(z/2)} \textnormal{d} z \textnormal{d} y\\
&=\frac{1}{8\pi^2}\textnormal{p.v.}\int_\mathbb{T}\textnormal{p.v.}\int_\mathbb{T}\frac{u(x)-u(x-y)}{\sin^2(y/2)}\frac{\partial_x u(x-y-z)}{\tan(z/2)} \textnormal{d} z \textnormal{d} y\\
&\quad-\frac{1}{4\pi^2}\textnormal{p.v.}\int_\mathbb{T}\textnormal{p.v.}\int_\mathbb{T}\frac{u(x)-u(x-y)}{\tan(y/2)}\frac{\partial_x\partial_y u(x-y-z)}{\tan(z/2)} \textnormal{d} z \textnormal{d} y\\
&=\frac{1}{8\pi^2}\textnormal{p.v.}\int_\mathbb{T}\textnormal{p.v.}\int_\mathbb{T}\frac{u(x)-u(x-y)}{\sin^2(y/2)}\frac{\partial_x u(x-y-z)}{\tan(z/2)} \textnormal{d} z \textnormal{d} y\\
&\quad+\frac{1}{4\pi^2}\textnormal{p.v.}\int_\mathbb{T}\textnormal{p.v.}\int_\mathbb{T}\frac{u(x)-u(x-y)}{\tan(y/2)}\frac{\partial_x^2 u(x-y-z)}{\tan(z/2)} \textnormal{d} z \textnormal{d} y\\
&=\frac{1}{8\pi^2}\textnormal{p.v.}\int_\mathbb{T}\textnormal{p.v.}\int_\mathbb{T}\frac{u(x)-u(x-y)}{\sin^2(y/2)}\frac{-\partial_z u(x-y-z)}{\tan(z/2)} \textnormal{d} z \textnormal{d} y-I_3\\
&=\frac{1}{16\pi^2}\textnormal{p.v.}\int_\mathbb{T}\textnormal{p.v.}\int_\mathbb{T}\frac{u(x)-u(x-y)}{\sin^2(y/2)}\frac{u(x-y)-u(x-y-z)}{\sin^2(z/2)} \textnormal{d} z \textnormal{d} y-I_3.
\end{align*}
Using that, for zero mean functions,
$$
\mathcal{H}(\mathcal{H}g)=-g,
$$
together with \eqref{auxiliar}, we find that 
$$
I_2=(\partial_x u(x))^2.
$$
Thus, collecting every term, we conclude the result.
\end{proof}
\begin{proof}[Proof of Proposition \ref{thm:Decay}]
Define $X_t$ such that
$$
M(t)=f(X_t,t)=\max_{y}f(y,t).
$$
We observe that $M(t)$ is $a.e.$ differentiable. To see that take two $t \le s$ and assume that $M (t) \ge M(s)$, then
\[
|M (t) - M(s)| = M (t) - M(s) = f(X_t,t) - f(X_s,s) \le f(X_t,t) \pm f(X_t,s)- f(X_s,s)\leqslant f(X_t,t) - f(X_t,s).
\]
As a consequence, the boundedness of $\norm{\partial_t f}_{L^\infty}$ is a sufficient condition for $M(t)$ to be Lipschitz. Now Rademacher Theorem gives us the a.e. differentiability of $M$. Furthermore, one can prove that the derivative verifies the following equality
$$
\ddt M(t)=\partial_t f(X_t,t)
$$
almost everywhere in time. Using Lemma \ref{lemma1}, we find that
$$
\ddt M(t)= \frac{1}{4\pi}\int_{\mathbb{T}}\frac{M(t)-u(X_t-y)}{\sin^2(y/2)}\left(\Lambda u(X_t-y)-1\right)\textnormal{d} y.
$$
Using that
$$
|\Lambda u|\leqslant \norm{u(t)}_{A^1},
$$
we conclude 
$$
\max_{y}f(y,t)\leqslant \max_{y}f_0(y).
$$
If we define $x_t$ such that
$$
m(t)=f(x_t,t)=\min_{y}f(y,t),
$$
we can repeat the previous steps and conclude the result.
\end{proof}

\subsection{Proof of Theorem \ref{thm:GWP_A1}}
\subsubsection*{Step 1: \emph{a priori estimates}} We start the proof providing the appropriate \emph{a priori} estimates in $A^1$. We have that
$$
\ddt \norm{f(t)}_{A^1}+\norm{f(t)}_{A^2}\leqslant\norm{\partial_{x}^2 \left(\comm{\mathcal{H}}{f}\Lambda f\right)}_{A^0}.
$$
In Fourier variables, we have that
\begin{align*}
\widehat{\partial_{x}^2 \left(\comm{\mathcal{H}}{f}\Lambda f\right)}&=\widehat{\partial_x \Lambda(f\Lambda f)+\partial_x^2(f\partial_x f)}\\
&=\hat{f}(k-m)\hat{f}(m)p(k,m),
\end{align*}
with
$$
p(k,m)=ik\left(|k||k-m|-k(k-m)\right)=ik|k||k-m|\left(1-\frac{k(k-m)}{|k||k-m|}\right).
$$
We observe that $p\neq0$ if and only if $0<|k|<|m|$. We find that
$$
|p(k,m)|\leqslant |k|2|k||k-m|\leq 2|m|^2|k-m|.
$$
Thus,
\begin{align*}
\norm{\partial_{x}^2 \left(\comm{\mathcal{H}}{f}\Lambda f\right)}_{A^0}&\leqslant 2\norm{f}_{A^1}\norm{f}_{A^2}.
\end{align*}
As a consequence, we find the inequality
$$
\ddt \norm{f(t)}_{A^1}+\norm{f(t)}_{A^2}\left(1-2\norm{f(t)}_{A^1}\right)\leqslant 0.
$$

\subsubsection*{Step 2: Approximated solutions and passing to the limit} 

To construct solutions we consider the regularized problem where the initial data is localized in Fourier space. In other words, given $f_0$, we consider \eqref{eq:AD0} with the initial data $f_0^N$ where
$$
f_0^N=\sum_{n=-N}^N\hat{f_0}(n)e^{inx}.
$$
Then, invoking Theorem \ref{thm:analytic}, there exists a local analytic solution which also enjoys continuity in time. For this local solution we can apply the previous energy estimates and, consequently, we can pass to the limit.

\subsubsection*{Step 3: Uniqueness} 

We argue by contradiction: let's assume that there exist two different solutions $f_1$ and $f_2$ starting from the same initial data. Then, one can prove that the difference $g=f_1-f_2$ satisfies
$$
\ddt \|g\|_{A^0}+\|g\|_{A^1}\leq \|g\|_{A^1}(\|f_1\|_{A^1}+\|f_2\|_{A^1}). 
$$
Then, one concludes the uniqueness using Gronwall's inequality.

\subsubsection*{Step 4: Decay} We have that
$$
\ddt \norm{f(t)}_{A^0}+\norm{f(t)}_{A^1}\leqslant\norm{\partial_{x} \left(\comm{\mathcal{H}}{f}\Lambda f\right)}_{A^0}.
$$
Recalling \eqref{commFourier}, we find that
\begin{align*}
\norm{\partial_{x} \left(\comm{\mathcal{H}}{f}\Lambda f\right)}_{A^1}&\leqslant 2\norm{f}_{A^1}^2.
\end{align*}
As a consequence, we find the inequality
$$
\ddt \norm{f(t)}_{A^0}+\norm{f(t)}_{A^1}\left(1-2\norm{f(t)}_{A^1}\right)\leqslant 0.
$$
Using the Poincar\'e-like inequality
$$
\norm{f(t)}_{A^0}\leqslant \norm{f(t)}_{A^1},
$$
we find the exponential decay
$$
\norm{f(t)}_{A^0}\leqslant \norm{f_0}_{A^0}e^{-\left(1-2\norm{f_0}_{A^1}\right)t}.
$$

\subsection{Proof of Theorem \ref{thm:GWP_H3/2}}
\subsubsection*{Step 1: a priori estimates}
Equipped with the local existence of smooth solution emanating from an analytic initial data, the proof of  Theorem \ref{thm:GWP_H3/2} reduces to the derivation of appropriate energy estimates in the right Sobolev spaces. Thus, we consider $ f $ to be a (space-time) smooth solution of the periodic problem \eqref{eq:AD0}.

\subsubsection*{ Step 1.1: $ L^2 $ estimates.}
As the main nonlinear cancellation is already present at the $ L^2\pare{\bT} $ level, for the sake of clarity, we start with the global estimate for this low energy norm. Thus, our first goal is to prove that, if $\norm{f \pare{t}}_{H^{3/2}\pare{\bT}}$ is small enough, then
\begin{equation}\label{eq:L2_enest}
\norm{f\pare{t}}_{L^2\pare{\bT}}^2 + \int_0^t \norm{\Lambda^{1/2} f\pare{s}}_{L^2\pare{\bT}}^2 \dt \leqslant \norm{f_0}_{L^2 \pare{\bT}}^2. 
\end{equation}
Let us multiply equation \eqref{eq:AD0} for $ f $ and integrate by parts. Then, we deduce
\begin{equation}\label{eq:L2_enest1}
\begin{aligned}
\frac{1}{2} \ddt
 \norm{f\pare{t}}_{L^2\pare{\bT}}^2 + \norm{\Lambda^{1/2} f\pare{t}}_{L^2\pare{\bT}}^2 & = \int _{\bT} f \pare{\pare{\Lambda f}^2 - \pare{\partial_x f}^2} \dx, \\
 & =  \int _{\bT}f \ \pare{\cH-1}\partial_x f \ \pare{\cH +1}\partial_x f \ \dx, \\
 & = \cN \pare{f, f, f}, 
\end{aligned}
\end{equation}
where 
\begin{equation}\label{eq:def_cN}
{ \cN \pare{g_1, g_2, g_3}} = \int _{\bT}g_1 \ \pare{\cH-1}\partial_x g_2 \ \pare{\cH +1}\partial_x g_3 \ \dx. 
\end{equation}
In order to conclude the inequality \eqref{eq:L2_enest}, we have to prove that following inequality holds true
\begin{equation}\label{lem:commutation}
\av{\cN\pare{g, h, h}} \leqslant C \norm{g}_{H^{3/2}} \norm{h}_{H^{1/2}}^2. 
\end{equation}
Indeed, if we denote as $ \pare{\hat{g}_n}_{n\in\bZ}, \ \pare{\hat{h}_n}_{n\in\bZ} $ the Fourier transform of $ g $ and $ h $ respectively and use Plancherel Theorem, we can rewrite $ \cN\pare{g, h, h} $ as\begin{equation*}
\begin{aligned}
\cN\pare{g, h, h} & =
 \sum_{\substack{n\in\bZ \\ k+m=n}} \overline{ \hat{g}_n} \pare{-i \sgn m -1} \pare{i m} \hat{h}_m \ \pare{-i \sgn k +1} \pare{ik}\hat{h}_k , \\
& =
 \sum_{\substack{n\in\bZ \\ k+m=n}} \overline{ \hat{g}_n}  m k \ \hat{h}_m \hat{h}_k \pare{1+ \sgn m \ \sgn k}. 
\end{aligned}
\end{equation*}
We observe that, if the addends are nonzero, $k$ and $n$ must verify
\begin{align*}
0 \leqslant k \leqslant n &&\text{or} && n \leqslant k \leqslant 0. 
\end{align*}
In other words, 
$$
|k|\leq |n|.
$$
Furthermore, using the previous cancellation to find a symmetry in the series, we find that
\begin{equation}\label{eq:cN_simplified}
\cN\pare{g, h, h} =  4 \sum_{0\leqslant k \leqslant n} \pare{n-k} k \ \textnormal{Re}\pare{\overline{ \hat{g}_n} \hat{h}_{n-k}\hat{h}_k}. 
\end{equation}
Thus,
\begin{equation*}
\av{\cN\pare{g, h, h} } \leqslant 4\sum_{0\leqslant k \leqslant n} \av{n-k} |k| \ \av{\hat{g}_n} \av{\hat{h}_{n-k}} \av{\hat{h}_k}, 
\end{equation*}
but since $ 0\leqslant k\leqslant n $ we deduce that 
$$ 
\pare{n-k}k\leqslant n^{3/2}\pare{n-k}^{1/4} k^{1/4},
$$ from where, using the $\ell^2-\ell^2$ Cauchy-Schwartz inequality, we can obtain the following
\begin{equation}\label{eq:cN_ineq1}
\begin{aligned}
\av{\cN\pare{g, h, h} } & \leqslant 4 \pare{\sum_n n^3 \av{\hat{g}_n}^2}^{1/2} \pare{\sum_n \av{\sum_k \pare{n-k}^{1/4}\av{\hat{h}_{n-k}}\ k^{1/4}\av{\hat{h}_{k}}}^2}^{1/2}. 
\end{aligned}
\end{equation}
Let us define the following distribution
\begin{equation*}
H \pare{x} = \cF^{-1}\pare{ \pare{\av{\hat{h}_n}}_n }\pare{x}. 
\end{equation*}
We observe that, as long as $ \pare{\av{\hat{h}_n}}_n\in \ell^2 $, $ H $ is a $ L^2 $ function. We prove now that if $ h\in H^{3/2} $ then $  \pare{\av{\hat{h}_n}}_n\in \ell^1  $, so the embedding $ \ell^1\hra \ell^2 $ closes the argument. Indeed
\begin{equation*}
\begin{aligned}
\sum_n \av{\hat{h}_n} & \leqslant \pare{\sum_n \av{n}^3 \av{\hat{h}_n}^2}^{1/2} \pare{\sum_n \av{n}^{-3}}^{1/2}, \\
& \leqslant C \norm{h}_{H^{3/2}}. 
\end{aligned}
\end{equation*}
Using the auxiliary distribution $ H $ we can rewrite \eqref{eq:cN_ineq1} as
\begin{equation}\label{eq:cN_ineq2}
\av{\cN\pare{g, h, h}} \leqslant 4 \norm{g}_{H^{3/2}} \norm{\Lambda^{1/4} H \ \Lambda^{1/4} H}_{L^2}. 
\end{equation}
Using H\"older inequality and Sobolev embeddings we deduce that
\begin{equation*}
\norm{\Lambda^{1/4} H \ \Lambda^{1/4} H}_{L^2}\leqslant \norm{H}_{\dot{W}^{\frac{1}{4}, 4}}^2 \leqslant C \norm{H}_{H^{1/2}}^2, 
\end{equation*}
but
\begin{equation*}
\norm{H}_{H^{1/2}}^2 = \sum_n \av{n} \av{\hat{H}_n}^2= \sum_n \av{n} \av{\hat{h}_n}^2 = \norm{h}_{H^{1/2}}^2,
\end{equation*}
whence we deduce from \eqref{eq:cN_ineq2} that
\begin{equation*}
\av{\cN\pare{g, h, h}} \leqslant C \norm{g}_{H^{3/2}}\norm{h}_{H^{1/2}}^2. 
\end{equation*}
As a consequence, we find that \eqref{eq:L2_enest1} now reads
\begin{equation}\label{eq:L2_enest2}
\begin{aligned}
\frac{1}{2} \ddt
 \norm{f\pare{t}}_{L^2\pare{\bT}}^2 + \norm{\Lambda^{1/2} f\pare{t}}_{L^2\pare{\bT}}^2 & \leq C \norm{f}_{H^{3/2}}\norm{f}_{H^{1/2}}^2, 
\end{aligned}
\end{equation}
from where we find the desider inequality \eqref{eq:L2_enest}.

We need now to find the appropriate bound for the higher order energy given by the $H^{3/2}$ norm. 

\subsubsection*{ Step 1.2: dyadic estimates.} Let us apply the truncation operator $ \tqS $ to the equation \eqref{eq:AD0}. We then multiply the resulting equation for $ \tqS f $ and integrate in space to obtain: 
\begin{equation}\label{eq:evolution_dyadic_AD0}
\frac{1}{2} \ddt \norm{\tqS f}_{L^2}^2  + \norm{\tqS \Lambda^{1/2} f}_{L^2}^2  = A_q-B_q, 
\end{equation}
with
$$
A_q=\int \tqS \pare{f \ \Lambda f} \tqS \Lambda f \dx,
$$
$$
B_q=\int\tqS \pare{f \ \partial_x f }\tqS \partial_x f \ \dx.
$$
Using now Bony decomposition \eqref{bony decomposition asymmetric}, we can split the previous terms as follows
\begin{align*}
A_q & = T^A_{1, q} + T^A_{2, q } + T^A_{3, q } + R^A_q, \\
B_q & = T^B_{1, q} + T^B_{2, q } + T^B_{3, q } + R^B_q,
\end{align*}
where
\begin{align*}
T^A_{1, q} & = \int S_{q-1} f \ \pare{\tqS \Lambda f}^2 \dx, &
T^B_{1, q} & = \int S_{q-1} f \ \pare{\tqS \partial_x f}^2 \dx, \\
T^A_{2, q} & = \sumf \int \bra{\tqS, \SQ f}\triangle_{q'} \Lambda f \  \tqS \Lambda f \ \dx, & 
T^B_{2, q} & = \sumf \int \bra{\tqS, \SQ f}\triangle_{q'} \partial_x f \ \tqS \partial_x f \ \dx, \\
T^A_{3, q} & = \sumf \int \pare{\SQ - \Sq}f \ \tqS \triangle_{q'} \Lambda f \ \tqS \Lambda f \ \dx, &
T^B_{3, q} & = \sumf \int \pare{\SQ - \Sq}f \ \tqS \triangle_{q'} \partial_x f \ \tqS \partial_x f\ \dx,\\
R^A_q & = \sumi \int \tqS \pare{\tQS f \ S_{q'+2}\Lambda f} \ \tqS \Lambda f \dx, &
R^B_q & = \sumi \int \tqS \pare{\tQS f \ S_{q'+2}\partial_x f} \ \tqS \partial_x f \dx. 
\end{align*}

We have to estimate these eight terms. The terms $ T^A_1 $ and $ T^B_1 $ are the more singular. We will make use of the commutation properties \eqref{lem:commutation} to estimate such terms. First of all we remark that
\begin{equation*}
T^A_{1, q} - T^B_{1, q} = \cN\pare{S_{q-1} f, \tqS  f, \tqS  f},
\end{equation*}
where the trilinear operator $ \cN $ is defined in \eqref{eq:def_cN}. We can hence use inequality \eqref{lem:commutation} and \eqref{regularity_dyadic} in order to deduce 
\begin{equation}\label{eq:BonyTA1-TA2}
\begin{aligned}
\av{T^A_{1, q} - T^B_{1, q}} & \leqslant C \norm{f}_{H^{3/2}} \norm{\tqS\Lambda^{1/2} f}_{L^2}^2, \\
& \leqslant  C b_q 2^{-2qs} \norm{f}_{H^{3/2}} \norm{ \Lambda^{1/2} f}_{H^{s}}^2, 
\end{aligned}
\end{equation}
for any $ s\geqslant 3/2 $. 

Next we handle the remainder terms $ R^A_q, R^B_q $. Indeed, we have that 
\begin{equation*}
\begin{aligned}
R^A_q & = \sumi \int \tqS \pare{\tQS f \ S_{q'+2}\Lambda f} \ \tqS \Lambda f \dx, \\
& \leqslant \sumi \norm{\tQS f }_{L^{\frac{2p}{	p-2}}} \norm{S_{q'+2} \Lambda f}_{L^p} \norm{\tqS \Lambda f}_{L^2}. 
\end{aligned}
\end{equation*}
Since $ H^{\frac{1}{2}} \pare{\bT}\hra L^p\pare{\bT} $ for any $ p\in\left[2, \infty \right) $ we can say that
\begin{equation*}
\norm{S_{q'+2} \Lambda f}_{L^p} \lesssim \sqrt{p} \norm{f}_{H^{3/2}}. 
\end{equation*}
Moreover due to the localization in the Fourier space and the fact that $ q'>q-4 $ we deduce 
\begin{equation*}
\begin{aligned}
\norm{\tQS f }_{L^{\frac{2p}{	p-2}}}  \norm{\tqS \Lambda f}_{L^2} & \lesssim \norm{\tQS f }_{L^{\frac{2p}{	p-2}}} 2^{q/2} \norm{\tqS \Lambda^{1/2} f}_{L^2}, \\
& \lesssim 2^{q'/2}\norm{\tQS f }_{L^{\frac{2p}{	p-2}}}  \norm{\tqS \Lambda^{1/2} f}_{L^2}, \\
& \lesssim \norm{\tQS \Lambda^{1/2} f }_{L^{\frac{2p}{	p-2}}}  \norm{\tqS \Lambda^{1/2} f}_{L^2}.
\end{aligned}
\end{equation*}
 Using \eqref{regularity_dyadic} we deduce that
\begin{equation}\label{eq:BonyRA-1}
R^A_q \lesssim b_q 2^{-2 q s}\sqrt{p} \norm{f}_{H^{3/2}}\norm{\Lambda^{1/2}f}_{H^{s}} \norm{\Lambda^{1/2}f}_{\dot{W}^{s, \frac{2p}{	p-2} }}. 
\end{equation}

Using the Sobolev inequality \eqref{eq:Sobolev_emb_with_constant}, we can also obtain the following estimate for the term $ R^A_q $: 

\begin{equation}
\label{eq:BonyRA-2}
\begin{aligned}
R^A_q & = \sumi \norm{S_{q'+2} \Lambda f}_{L^\infty}  \norm{\tQS \Lambda^{1/2} f }_{L^{2}}  \norm{\tqS \Lambda^{1/2} f}_{L^2}, \\
& \lesssim b_q 2^{-2qs}  \frac{1}{\delta}  \norm{f}_{H^{\frac{3}{2} + \delta}}\norm{\Lambda^{1/2}f}_{H^s}^2,
\end{aligned}
\end{equation}
for any $ 0<\delta \ll 1 $.

A very similar procedure allows us to deduce the following bound
\begin{align}\label{eq:BonyRB-1}
R^B_q & \lesssim b_q 2^{-2 q s} \sqrt{p} \norm{f}_{H^{3/2}}\norm{\Lambda^{1/2}f}_{H^{s}} \norm{\Lambda^{1/2}f}_{\dot{W}^{s, \frac{2p}{	p-2} }}, \\
R^B_q &\lesssim b_q 2^{-2qs}\frac{1}{\delta}   \norm{f}_{H^{\frac{3}{2} + \delta}}\norm{\Lambda^{1/2}f}_{H^s}^2. \label{eq:BonyRB-2}
\end{align}

We turn now our attention to the element $ T^A_2 $. Using Lemma \ref{estimates commutator} we obtain
\begin{equation}
\label{eq:BonyTA2-1}
\begin{aligned}
\av{T^A_{2, q}} & \leq  \sumf \norm{\bra{\tqS, S_{q'-1}f} \triangle_{q'} \Lambda f}_{L^2} \norm{\tqS\Lambda f}_{L^2}, \\
& \leq   \sumf 2^{-q} \norm{\tqS \Lambda f}_{L^2} \norm{\Lambda f}_{L^p}\norm{\triangle_{q'}\Lambda f}_{L^{\frac{2p}{p-2}}}, \\
&\lesssim \sumf  \norm{\tqS \Lambda^{1/2} f}_{L^2} \norm{\Lambda f}_{L^p}\norm{\triangle_{q'}\Lambda^{1/2} f}_{L^{\frac{2p}{p-2}}}, \\
& \lesssim b_q 2^{-2 q s} \sqrt{p} \norm{f}_{H^{3/2}}\norm{\Lambda^{1/2}f}_{H^{s}} \norm{\Lambda^{1/2}f}_{\dot{W}^{ s , \frac{2p}{	p-2} }}. 
\end{aligned}
\end{equation}
As a consequence, we can conclude the following bound:
\begin{equation}\label{eq:BonyTA2-2}
\av{T^A_{2, q}} \lesssim b_q 2^{-2qs} \frac{1}{\delta}   \norm{f}_{H^{\frac{3}{2} + \delta}} \norm{\Lambda^{1/2}f}_{H^s}^2. 
\end{equation}

Similarly, we can deduce the bounds
\begin{align}\label{eq:BonyTB2-1}
\av{T^B_{2, q}} & \lesssim b_q 2^{-2 q s} \sqrt{p} \norm{f}_{H^{3/2}}\norm{\Lambda^{1/2}f}_{H^{s}} \norm{\Lambda^{1/2}f}_{\dot{W}^{ s , \frac{2p}{	p-2} }}, \\
\av{T^B_{2, q}} & \lesssim b_q 2^{-2qs} \frac{1}{\delta}   \norm{f}_{H^{\frac{3}{2} + \delta}} \norm{\Lambda^{1/2}f}_{H^s}^2. \label{eq:BonyTB2-2}
\end{align}

The estimates 
\begin{equation}\label{eq:BonyTA3}
\av{T^A_{3, q}} + \av{T^B_{3, q}} \lesssim
b_q 2^{-2 q s} \norm{f}_{H^{3/2}}\norm{\Lambda^{1/2}f}_{H^{s}}^2, 
\end{equation}
can be easily obtained since the terms composing the elements $ T^A_{3, q}, \ T^B_{3, q} $ are all localized in dyadic annuli, and hence such terms are more regular than the ones estimated above. We note that
\begin{equation*}
\frac{1}{2} \ddt \norm{\tqS f}_{L^2}^2  + \norm{\tqS \Lambda^{1/2} f}_{L^2}^2  \leqslant C b_q 2^{-2 q s} \bra{ \norm{f}_{H^{3/2}}\norm{\Lambda^{1/2}f}_{H^{s}}^2 \\
+  \sqrt{p} \norm{f}_{H^{3/2}}\norm{\Lambda^{1/2}f}_{H^{s}} \norm{\Lambda^{1/2}f}_{\dot{W}^{ s , \frac{2p}{p-2} }}},
\end{equation*}
where we have used \eqref{eq:BonyTA1-TA2}, \eqref{eq:BonyRA-1}, \eqref{eq:BonyRB-1}, \eqref{eq:BonyTA2-1}, \eqref{eq:BonyTB2-1} and \eqref{eq:BonyTA3}. Now, multiplying the above equation for $ 2^{2qs} $ and summing in $ q\in \bZ $, we obtain
\begin{equation}\label{eq:Hs_est1}
\frac{1}{2} \ddt \norm{f}_{\Hs}^2 + \norm{\Lambda^{1/2} f }_{\Hs}^2 \leqslant C  \bra{ \norm{f}_{H^{3/2}}\norm{\Lambda^{1/2}f}_{H^{s}}^2 
+  \sqrt{p} \norm{f}_{H^{3/2}}\norm{\Lambda^{1/2}f}_{H^{s}} \norm{\Lambda^{1/2}f}_{\dot{W}^{ s , \frac{2p}{p-2} }}} .
\end{equation}
Using \eqref{eq:BonyTA1-TA2}, \eqref{eq:BonyRA-2}, \eqref{eq:BonyRB-2}, \eqref{eq:BonyTA2-2}, \eqref{eq:BonyTB2-2} and \eqref{eq:BonyTA3} instead we deduce
\begin{equation}\label{eq:Hs_est2}
\frac{1}{2} \ddt \norm{f}_{\Hs}^2 + \norm{\Lambda^{1/2} f }_{\Hs}^2 \leqslant C  \bra{ \norm{f}_{H^{3/2}} 
+  \frac{1}{\delta}  \norm{f}_{H^{\frac{3}{2} + \delta}} \Big.  } \norm{\Lambda^{1/2}f}_{H^s}^2 .
\end{equation}

\subsubsection*{Step 1.3: $ H^{3/2} $ estimates} We can now simply consider the estimate \eqref{eq:Hs_est1} in which we set $ s=3/2 $ and
\begin{equation*}
p = p_\varepsilon =\frac{1}{\varepsilon}, 
\end{equation*}
in order to deduce
\begin{equation*}
\frac{1}{2} \ddt \norm{f}_{ H^{3/2} }^2 + \norm{\Lambda^{1/2} f }_{ H^{3/2}}^2 \leqslant C  \pare{ \norm{f}_{H^{3/2}}\norm{\Lambda^{1/2}f}_{ H^{3/2}}^2 
+  \frac{1}{\sqrt{\varepsilon}} \norm{f}_{H^{3/2}}\norm{\Lambda^{1/2}f}_{ H^{3/2}} \norm{\Lambda^{1/2}f}_{\dot{W}^{ \frac{3}{2} , \frac{2p}{p-2} }}}. 
\end{equation*}
Next we remark that $ H^{\frac{3}{2} + \varepsilon }\hra\dot{W}^{ \frac{3}{2} , \frac{2p}{p-2} }$, whence
\begin{equation}\label{eq:H3/2_est_final}
\frac{1}{2} \ddt \norm{f}_{ H^{3/2} }^2 + \norm{\Lambda^{1/2} f }_{ H^{3/2}}^2 \leqslant C  \pare{ \norm{f}_{H^{3/2}}\norm{\Lambda^{1/2}f}_{ H^{3/2}}^2 
+  \frac{1}{\sqrt{\varepsilon}} \norm{f}_{H^{3/2}}\norm{\Lambda^{1/2}f}_{ H^{3/2}} \norm{\Lambda^{1/2}f}_{ H^{\frac{3}{2} + \varepsilon} }}. 
\end{equation}

\subsubsection*{Step 1.4: $ H^{\frac{3}{2} + \varepsilon} $ estimates} We invoke the estimate \eqref{eq:Hs_est2} with $ s=\frac{3}{2} + \varepsilon $ and $ 0<\delta\ll \varepsilon $:
\begin{equation*}
\frac{1}{2} \ddt \norm{f}_{H^{\frac{3}{2} + \varepsilon}}^2 + \norm{\Lambda^{1/2} f }_{H^{\frac{3}{2} + \varepsilon}}^2 \leqslant C  \pare{ \norm{f}_{H^{3/2}} 
+   \frac{1}{\delta} \norm{f}_{H^{\frac{3}{2} + \delta}}} \norm{\Lambda^{1/2}f}_{H^{\frac{3}{2} + \varepsilon}}^2 .  
\end{equation*}
Using interpolation between Sobolev spaces we deduce
\begin{equation*}
\norm{f}_{H^{\frac{3}{2} + \delta}} \leqslant \norm{f}_{H^{3/2}}^{1-\frac{\delta}{\varepsilon}} \norm{f}_{H^{\frac{3}{2} + \varepsilon}}^{\frac{\delta}{\varepsilon}}, 
\end{equation*}
from where we obtain that
\begin{equation}\label{eq:H3/2+A_est_final}
\frac{1}{2} \ddt \norm{f}_{H^{\frac{3}{2} + \varepsilon}}^2 + \norm{\Lambda^{1/2} f }_{H^{\frac{3}{2} + \varepsilon}}^2 \leqslant C  \pare{ \norm{f}_{H^{3/2}} 
+  \frac{1}{\delta}   \norm{f}_{H^{3/2}}^{1-\frac{\delta}{\varepsilon}} \norm{f}_{H^{\frac{3}{2} + \varepsilon}}^{\frac{\delta}{\varepsilon}} } \norm{\Lambda^{1/2}f}_{H^{\frac{3}{2} + \varepsilon}}^2 . 
\end{equation}

\subsubsection*{Step 1.5: closing the estimates} In this section we conclude the high order energy estimates. We sum up \eqref{eq:H3/2_est_final} and \eqref{eq:H3/2+A_est_final} to find that
\begin{multline}\label{eq:sum_estimates_1}
\frac{1}{2} \ddt \pare{\norm{f}_{ H^{3/2} }^2 + \norm{f}_{H^{\frac{3}{2} + \varepsilon}}^2} + \pare{ \norm{\Lambda^{1/2} f }_{ H^{3/2}}^2+ \norm{\Lambda^{1/2} f }_{H^{\frac{3}{2} + \varepsilon}}^2} 
\\
\leqslant C  \pare{1+\frac{1}{\sqrt{\varepsilon}}} \norm{f}_{H^{3/2}}\norm{\Lambda^{1/2}f}_{ H^{3/2}}^2 
+ 
C \pare{ \pare{1+\frac{1}{\sqrt{\varepsilon}}} \norm{f}_{H^{3/2}} 
+ \frac{1}{\delta}   \norm{f}_{H^{3/2}}^{1-\frac{\delta}{\varepsilon}} \norm{f}_{H^{\frac{3}{2} + \varepsilon}}^{\frac{\delta}{\varepsilon}} } \norm{\Lambda^{1/2}f}_{H^{\frac{3}{2} + \varepsilon}}^2 . 
\end{multline}

We define
\begin{equation*}
\beta = \beta\pare{f_0} = \frac{\norm{f_0}_{H^{\frac{3}{2} + \varepsilon}}}{\norm{f_0}_{H^{\frac{3}{2}}}} > 1.
\end{equation*}

We want now to absorb the contribution of 
$$ 
\frac{1}{\delta}   \norm{f}_{H^{3/2}}^{1-\frac{\delta}{\varepsilon}} \norm{f}_{H^{\frac{3}{2} + \varepsilon}}^{\frac{\delta}{\varepsilon}}  \norm{\Lambda^{1/2}f}_{H^{\frac{3}{2} + \varepsilon}}^2 
$$ appearing in the right hand side of \eqref{eq:sum_estimates_1} in the parabolic term in the left hand side of the same estimate. In order to do so we must split the proof according to the explicit value of $ \beta $:
 \\

{\bf Case 1: $ \beta \leqslant 2 $.} In such setting we fix $ \delta =\varepsilon /2 $, hence $ \beta^{\delta/\varepsilon} / \delta \leqslant 2^{3/2} / \varepsilon $. Thus
\begin{equation}\label{eq:smallness_condition1}
\frac{1}{\delta}  \norm{f}_{H^{3/2}}^{1-\frac{\delta}{\varepsilon}} \norm{f}_{H^{\frac{3}{2} + \varepsilon}}^{\frac{\delta}{\varepsilon}} = \frac{1}{\delta} \norm{f}_{H^{\frac{3}{2}}} \beta^{\delta / \varepsilon} \leqslant \frac{2^{3/2}}{\varepsilon}  \norm{f}_{H^{\frac{3}{2}}} . 
\end{equation}
So, if the initial data is small enough in $H^{3/2}$, then \eqref{eq:smallness_condition1} is as small as needed in order to absorb the desired term.

{\bf Case 2: $ \beta > 2 $.} In such setting we can define
\begin{equation*}
\delta = \delta \pare{f_0} = \varepsilon \log_\beta 2 < \varepsilon. 
\end{equation*}
The value of $ \delta $ has been defined so that
\begin{equation*}
\beta^{\delta / \varepsilon} = 2.
\end{equation*}
We compute
\begin{equation*}
\frac{1}{\delta}  \norm{f}_{H^{3/2}}^{1-\frac{\delta}{\varepsilon}} \norm{f}_{H^{\frac{3}{2} + \varepsilon}}^{\frac{\delta}{\varepsilon}}=
\frac{1}{\delta} \norm{f_0}_{H^{\frac{3}{2}}} \beta^{\delta / \varepsilon}  = \frac{2  }{\varepsilon \log_\beta 2 } \ \norm{f_0}_{H^{\frac{3}{2} }}, 
\end{equation*}
We want to make sure that
$$ 
\frac{2 C  }{\varepsilon \log_\beta 2 } \ \norm{f_0}_{H^{\frac{3}{2}}} < 1/2. 
$$ 
Thus, simplifying the previous expression, we obtain that our goal is to prove that the following inequality holds
\begin{equation*}
\begin{aligned}
  \norm{f_0}_{H^{\frac{3}{2}}} \log \pare{ \frac{\norm{f_0}_{H^{\frac{3}{2}+\varepsilon}} }{\norm{f_0}_{H^{\frac{3}{2}}}}
} 
=  \norm{f_0}_{H^{\frac{3}{2}}}\pare{ \log \norm{f_0}_{H^{\frac{3}{2}+\varepsilon}}  -\log \norm{f_0}_{H^{\frac{3}{2}}}
} 
 < \frac{\varepsilon\log 2}{4C}  .
\end{aligned}
\end{equation*}
We will now use the algebraic inequality $ -x \log x \leq \sqrt{x},\ x\in\bra{0, 1} $ to obtain
\begin{equation*}
\begin{aligned}
\norm{f_0}_{H^{\frac{3}{2}}} \log \pare{ \frac{\norm{f_0}_{H^{\frac{3}{2}+\varepsilon}} }{\norm{f_0}_{H^{\frac{3}{2}}}}
} & =  \norm{f_0}_{H^{\frac{3}{2}}}\pare{ \log \norm{f_0}_{H^{\frac{3}{2}+\varepsilon}}  -\log \norm{f_0}_{H^{\frac{3}{2}}}
}, \\
  & \leq
\norm{f_0}_{H^{\frac{3}{2}}}^{1/2} \pare{\norm{f_0}_{H^{\frac{3}{2}+\varepsilon}}^{1/2} \log \norm{f_0}_{H^{\frac{3}{2}+\varepsilon}} +1}, \\
\end{aligned}
\end{equation*}
It is now clear that, under the assumptions of the statement, we have that
\begin{equation}
\label{eq:smallness_condition2}
\begin{aligned}
\norm{f_0}_{H^{\frac{3}{2}}}^{1/2} \pare{\norm{f_0}_{H^{\frac{3}{2}+\varepsilon}}^{1/2} \log \norm{f_0}_{H^{\frac{3}{2}+\varepsilon}} +1} < \frac{\varepsilon\log 2}{4C} .
\end{aligned}
\end{equation}
In other words, we can now collect the estimates \eqref{eq:smallness_condition1} and \eqref{eq:smallness_condition2} to obtain that if the initial data satisfies the condition
\begin{align*}
\norm{f_0}_{H^{\frac{3}{2}}} <
\min \set{1 \  , \frac{\varepsilon^2\log^2 2 }{16 C^2 \ \pare{\norm{f_0}_{H^{\frac{3}{2}+\varepsilon}} \log \norm{f_0}_{H^{\frac{3}{2}+\varepsilon}} +1}^2}  } , 
\end{align*}
then
\begin{equation*}
\frac{1}{\delta}   \norm{f}_{H^{3/2}}^{1-\frac{\delta}{\varepsilon}} \norm{f}_{H^{\frac{3}{2} + \varepsilon}}^{\frac{\delta}{\varepsilon}}  \norm{\Lambda^{1/2}f}_{H^{\frac{3}{2} + \varepsilon}}^2 < \frac{1}{2} \norm{\Lambda^{1/2}f}_{H^{\frac{3}{2} + \varepsilon}}^2.
\end{equation*}
Thus, due to the previous energy estimates, we can absorb the contribution from the nonlinear terms and conclude that
\begin{equation}\label{eq:est_H3/2}
 {\norm{f \pare{t} }_{ H^{3/2} }^2 + \norm{f\pare{t}}_{H^{\frac{3}{2} + \varepsilon}}^2} +\int_0^t \pare{ \norm{\Lambda^{1/2} f \pare{s} }_{ H^{3/2}}^2+ \norm{\Lambda^{1/2} f \pare{s}  }_{H^{\frac{3}{2} + \varepsilon}}^2} \dt \leqslant
 \norm{f _0 }_{ H^{3/2} }^2 + \norm{f_0}_{H^{\frac{3}{2} + \varepsilon}}^2. 
\end{equation}
The above equation implies hence that smooth solutions stemming from small $ H^{3/2} $ initial data are nonlinearly stable in  $H^{\frac{3}{2} } \cap H^{\frac{3}{2} + \varepsilon} $.

\subsubsection*{Step 1.6: estimates for the time derivative} The purpose of this section is to find estimates for the time derivative of the solution. In particular, we want to obtain the following
\begin{align}
 \label{prop:time_der_estimates}
&\int_0^t \norm{\partial_t f \pare{s} }_{ H^1}^2 ds
 \leqslant C \set{ 
\pare{\norm{f _0 }_{ H^{3/2} }^2 + \norm{f_0}_{H^{\frac{3}{2} + \varepsilon}}^2}\bra{ 1+ \pare{\norm{f _0 }_{ H^{3/2} }^2 + \norm{f_0}_{H^{\frac{3}{2} + \varepsilon}}^2} \pare{1+\frac{1}{\sqrt{\varepsilon}}}} + \norm{f _0 }_{ H^{3/2} }^3 + \norm{f_0}_{H^{\frac{3}{2} + \varepsilon}}^3
}.
\end{align}
We multiply \eqref{eq:AD0} by $ \partial_t \Lambda^2 f $ and integrate in $ \bT \times \pare{0, T} $ to obtain

\begin{equation}\label{eq:est_patf1}
\norm{\partial_t \Lambda f}_{L^2_T L^2}^2 + \pare{\norm{f\pare{t}}^2_{H^{\frac{3}{2}}} - \norm{f_0}^2_{H^{\frac{3}{2}}}} = 
I_1+I_2 .
\end{equation}
where
$$
I_1=\int_0^T\int_{\bT} \Lambda \pare{f\Lambda f} \partial_t \Lambda^2 f  \ \dx \dt,
$$
and
$$
I_2=\int_0^T\int_{\bT} \partial_x \pare{f\partial_x f} \partial_t \Lambda^2 f \ \dx \dt.
$$
Using the identity $ \Lambda^2 = -\partial_x^2 $ and integrating by parts we can deduce the following equalities
\begin{equation*}
\begin{aligned}
I_1 & = \int_0^T\int_{\bT} \partial_x f \Lambda f \ \partial_t \Lambda \partial_x f \ \dx\dt + \int_0^T\int_{\bT}  f \Lambda \partial_x f \  \partial_t \Lambda \partial_x f \ \dx\dt, \\
& = I_{1,1} +I_{1,2}, \\
I_2 & = - \int_0^T\int_{\bT}  \pare{\partial_x f}^2 \partial_t \partial_x^2 f \ \dx\dt - \int_0^T\int_{\bT}  f {\partial_x^2 f} \partial_t \partial_x^2 f \ \dx\dt, \\
& =  I_{2,1} +I_{2,2}. 
\end{aligned}
\end{equation*}

We hence obtained that

\begin{equation*}
\begin{aligned}
I_{1,2}+I_{2,2} & = \frac{1}{2} \int_0^T\int_{\bT} f \partial_t \bra{ \pare{\Lambda \partial_x f}^2 - \pare{\partial_x^2 f}^2} \dx \dt, \\
& = \left. \frac{1}{2} \  \cN \pare{f, \partial_x f, \partial_x f}  \Big.  \right|_0^T -  \frac{1}{2} \int_0^T \cN \pare{\partial_t f, \partial_x f, \partial_x f} \dt, \\
& = J_1 +J_2. 
\end{aligned}
\end{equation*}

The term $ J_1 $ can be estimated using \eqref{lem:commutation} and we deduce
\begin{equation}\label{eq:est_C1}
J_1 \lesssim \norm{f\pare{T}}_{H^{3/2}}^3 + \norm{f_0}_{H^{3/2}}^3 .
\end{equation}

The term $ J_2 $ can be estimated in a similar way. In particular, we obtain for any $ \epsilon > 0 $
\begin{equation}\label{eq:est_C2}
\begin{aligned}
J_2 & \lesssim \int_0^T \norm{\partial_t f }_{H^1} \norm{\Lambda^{3/2} f }_{L^4}^2 \dt, \\
& \lesssim \int_0^T \norm{\partial_t f }_{H^1} \norm{  f }_{H^{3/2}} \norm{\Lambda^{1/2} f }_{H^{3/2}} \dt, \\
& \leqslant \epsilon \norm{\partial_t f}_{L^2_T H^1}^2 + \frac{C}{\epsilon} \norm{f}_{L^\infty_T H^{3/2}}^2 \norm{ \Lambda^{1/2} f}_{L^2 _T H^{3/2}}^2 .
\end{aligned}
\end{equation}

We have to estimate the term $I_{1,1}+I_{2,1}$. We find that

\begin{equation}\label{eq:est_A1}
\begin{aligned}
I_{1,1} & =\int_0^T\int_{\bT} \partial_x f \Lambda f \ \partial_t \Lambda \partial_x f \ \dx\dt, \\
& \sim \int_0^T\int_{\bT} \Lambda^2 f \Lambda f \ \partial_t \Lambda f \ \dx\dt, \\
& \lesssim \norm{\Lambda^2 f}_{L^2_T L^{\frac{2p_\varepsilon}{p_\varepsilon-2}}} \norm{\Lambda f}_{L^\infty_T L^{p_\varepsilon}} \norm{\partial_t \Lambda f}_{L^2_T L^2}, \\
& \lesssim \frac{1}{\sqrt{\varepsilon}} \norm{\Lambda^{1/2} f }_{L^2_T H^{\frac{3}{2}+\varepsilon}} \norm{f}_{L^\infty_T H^{3/2}} \norm{\partial_t f }_{L^2_T H^1}, \\
& \leqslant \eta \norm{\partial_t f }_{L^2_T H^1}^2 + \frac{C}{\eta\sqrt{\varepsilon}} \norm{\Lambda^{1/2} f }_{L^2_T H^{\frac{3}{2}+\varepsilon}}^2 \norm{f}_{L^\infty_T H^{3/2}}^2,
\end{aligned}
\end{equation}
where $\eta>0$ is arbitrary. 
The term $I_{2,1}$ is completely analogous to $ I_{1, 1} $ and can be handled in a similar way and we obtain the estimate
\begin{equation*}
\begin{aligned}
I_{2,1} 
& \leqslant \eta \norm{\partial_t f }_{L^2_T H^1}^2 + \frac{C}{\eta\sqrt{\varepsilon}} \norm{\Lambda^{1/2} f }_{L^2_T H^{\frac{3}{2}+\varepsilon}}^2 \norm{f}_{L^\infty_T H^{3/2}}^2.
\end{aligned}
\end{equation*}
Using the estimates \eqref{eq:est_C1}, \eqref{eq:est_C2}, \eqref{eq:est_A1} in \eqref{eq:est_patf1} we obtain that
\begin{multline*}
\pare{1-2\eta} \norm{\partial_t f }_{L^2_T H^1}^2 \\
 \leqslant C \bra{ 
 \norm{f_0}^2_{H^{\frac{3}{2}}} - \norm{f\pare{T}}^2_{H^{\frac{3}{2}}} +
 \norm{f\pare{T}}_{H^{3/2}}^3 + \norm{f_0}_{H^{3/2}}^3
+ \frac{1}{\eta} 
\pare{
\norm{ \Lambda^{1/2} f}_{L^2 _T H^{3/2}}^2 + \frac{1}{\sqrt{\varepsilon}}\norm{\Lambda^{1/2} f }_{L^2_T H^{\frac{3}{2}+\varepsilon}}^2
}\norm{f}_{L^\infty_T H^{3/2}}^2
 } .
\end{multline*}
We fix $ \eta = 1/4 $ and use \eqref{eq:est_H3/2} in order to deduce the bound

\begin{equation*}
\frac{\norm{\partial_t f }_{L^2_T H^1}^2}{2} \leqslant C \set{ 
\pare{\norm{f _0 }_{ H^{3/2} }^2 + \norm{f_0}_{H^{\frac{3}{2} + \varepsilon}}^2}\bra{ 1+ \pare{\norm{f _0 }_{ H^{3/2} }^2 + \norm{f_0}_{H^{\frac{3}{2} + \varepsilon}}^2} \pare{1+\frac{1}{\sqrt{\varepsilon}}}} + \norm{f _0 }_{ H^{3/2} }^3 + \norm{f_0}_{H^{\frac{3}{2} + \varepsilon}}^3
}.
\end{equation*}

\subsubsection*{Step 2: Approximated solutions and passing to the limit} Using the apriori estimates provided before the proof of Theorem \ref{thm:GWP_H3/2} become a standard approximation argument which we outline here. Let us define the truncation operator
\begin{equation}\label{truncation}
\cJ_n u \pare{x} = \frac{1}{\sqrt{2\pi}}\sum_{\av{k}\leqslant n} \hat{u}_k e^{i \ k\cdot x} = \cF^{-1}\pare{\pare{1_{\set{\av{k}\leqslant n}}\pare{k} \hat{u}_k}_k}. 
\end{equation}
We apply this truncation operator to the initial data $f_0$. Then, this truncated initial data is analytic and, invoking Theorem \ref{thm:analytic}, we obtain the existence of smooth (in time and space) local solution approximate solutions. We can use the previous estimates \eqref{eq:est_H3/2} to deduce that the local approximate solutions are actually globally defined as $H^{3/2+\varepsilon}$ functions. The sequence  of approximate solutions $ \pare{f_n}_n $ is bounded uniformly in $n$ and $T$ in the spaces 
$$ 
L^\infty\pare{\left[ 0, T \right); H^{\frac{3}{2}}\cap H^{\frac{3}{2}+\varepsilon}} \cap L^2\pare{\left[ 0, T \right); H^{2} \cap H^{2+\varepsilon}}, 
$$
so there exists a weak limit
\begin{equation*}
f\in L^\infty \pare{\left[0, T \right) ; \ H^{\frac{3}{2}+ \varepsilon}} \cap  L^2 \pare{0, T ; \ H^{2+ \varepsilon}}. 
\end{equation*} 
Furthermore, estimate \eqref{prop:time_der_estimates} ensures that the sequence fo approximate solutions $ \pare{f_n}_n $ is uniformly bounded in 
$$ 
H^1\pare{\left[0, T\right); H^1},
$$ 
from where, using interpolation, we deduce that
 \begin{equation*}
 \pare{f_n}_n \text{ is uniformly bounded in } H^{1-\vartheta}\pare{\left[0, T\right) ; H^{1+\vartheta\pare{\varepsilon+1}}}, \ \vartheta\in \bra{0, 1}. 
 \end{equation*}
 Setting 
 \begin{equation}\label{eq:relation_vartheta}
   \vartheta  \in \left[ \frac{1}{2\pare{\varepsilon+1}} \ , \ \frac{1}{2} \ \right), 
 \end{equation}
 we deduce that
 \begin{equation*}
 \pare{f_n}_n \text{ is uniformly bounded in } C^{0, \frac{1}{2} -\vartheta} \pare{\left[0, T\right) ; H^{1+\vartheta\pare{\varepsilon+1}}}.
 \end{equation*}
We obtain that $ \pare{f_n}_n $ is equicontinuous-in-time. We can hence apply Ascoli-Arzel\'a theorem in order to deduce that there exists a $ f\in  C^{0, \frac{1}{2} -\vartheta} \pare{\left[0, T\right) ; H^{1+\vartheta\pare{\varepsilon+1}}}$ such that, taking a subsequence if necessary, $ f_n\to f $ in $C^{0, \frac{1}{2} -\vartheta} \pare{\left[0, T\right) ; H^{1+\vartheta\pare{\varepsilon+1}}} $ as long as $ \vartheta $ satisfies \eqref{eq:relation_vartheta}. It is a classical argument to prove that $ f $ solves \eqref{eq:AD0} and that it is unique.
 
\section{Gravity-capillarity driven system \eqref{eq:AD}}
\subsection{Proof of Theorem \ref{thm:GWP_A1surf}} 

\subsubsection*{Step 1: a priori estimates} We only prove the a priori estimates, since the approximation procedure is standard and it can be done using a Galerkin scheme. We have that
$$
\ddt \norm{f(t)}_{\dot{A}^1}+\nu\norm{f(t)}_{A^4}+\norm{f(t)}_{A^2}\leqslant\norm{\partial_{x}^2 \left(\comm{\mathcal{H}}{f}(\nu\Lambda^3f+\Lambda f)\right)}_{A^0}.
$$
As before,
\begin{align*}
\norm{\partial_{x}^2 \left(\comm{\mathcal{H}}{f}\Lambda f\right)}_{A^0}&\leqslant 2\norm{f}_{A^1}\norm{f}_{A^2}.
\end{align*}
In Fourier variables, we have that
\begin{align*}
\widehat{\partial_{x}^2 \left(\comm{\mathcal{H}}{f}\Lambda^3 f\right)}&=\widehat{\partial_x \Lambda(f\Lambda^3 f)-\partial_x^2(f\partial_x^3 f)}\\
&=\hat{f}(k-m)\hat{f}(m)p(k,m),
\end{align*}
with
\begin{align*}
p(k,m)&=ik\left(|k||k-m|^3-k(k-m)^3\right)=ik|k||k-m|^3\left(1-\frac{k(k-m)^3}{|k||k-m|^3}\right)
\end{align*}
Again $p\neq0$ if and only if $0<|k|<|m|$. We find that
$$
|p(k,m)|\leqslant 2|k||k||k-m|^3\leq 2|m|^2|k-m|^3.
$$
Thus, using the interpolation inequality (valid for $0\leq r\leq 3$)
$$
\|u\|_{A^r}\leqslant \|u\|_{A^0}^{1-r/3}\|u\|_{A^3}^{r/3},
$$
we find that 
\begin{align*}
\norm{\partial_{x}^2 \left(\comm{\mathcal{H}}{f}\Lambda^3 f\right)}_{A^0}&\leqslant 2\nu\norm{f}_{A^1}\norm{f}_{A^4}.
\end{align*}
As a consequence, we find the inequality
$$
\ddt \norm{f(t)}_{\dot{A}^1}+\nu\norm{f(t)}_{A^4}+\norm{f(t)}_{A^2}\leqslant 2\norm{f}_{A^1}\norm{f}_{A^2}+2\nu\norm{f}_{A^1}\norm{f}_{A^4},
$$
and we conclude the result.

\subsubsection*{Step 2: Decay} We have that
$$
\ddt \norm{f(t)}_{A^0}+\nu\norm{f(t)}_{A^3}+\norm{f(t)}_{A^1}\leqslant 2\norm{f}_{A^1}^2+2\nu\norm{f}_{A^1}\norm{f}_{A^3},
$$
thus,
$$
\ddt \norm{f(t)}_{A^0}+(\nu\norm{f(t)}_{A^3}+\norm{f(t)}_{A^1})\left(1-2\norm{f}_{A^1}\right)\leqslant 0.
$$
From here we conclude using a Poincar\'e type inequality.

\subsection{Proof of Theorem \ref{thm:GWP_H2}}
\subsubsection*{Step 1: a priori estimates} 
\subsubsection*{ Step 1.1: $ L^2 $ estimates.}
For the sake of clarity, we first perform $ L^2 $ energy estimates for \eqref{eq:AD} and we study the particular commutation properties of the nonlinearity of \eqref{eq:AD}. We multiply \eqref{eq:AD} by $ f $ and integrate in $ \bT $ to deduce the following 
\begin{equation*}
\frac{1}{2} \ddt \norm{f\pare{t}}_{L^2}^2 + \nu \norm{ \Lambda^{3/2} f\pare{t}}_{L^2}^2 + \norm{\Lambda^{1/2} f\pare{t}}_{L^2}^2 = \cN\pare{f, f, f} + \cM\pare{f, f, f}, 
\end{equation*}
where $ \cN $ is defined in \eqref{eq:def_cN} and $ \cM $ is defined as
\begin{equation}
\label{eq:def_cM}
\cM \pare{g_1, g_2, g_3} = \nu \int _{\bT} g_1 \pare{\Lambda^3 g_2 \  \Lambda g_3 + \partial_x^3 g_2 \ \partial_x g_3} \ \dx. 
\end{equation}

Our goal now is to prove that
\begin{equation}\label{lem:commutation_cM}
\av{\cM\pare{g, h, h}} \leqslant C \nu \norm{g}_{H^2\pare{\bT}}\norm{h}_{H^{1/2}\pare{\bT}}^{1/2} \norm{h}_{H^{3/2}\pare{\bT}}^{3/2}. 
\end{equation}
Indeed, using Pancherel theorem, we can write
\begin{equation}\label{eq:def_cM_fourier}
\cM \pare{g, h, h} = \nu \sum_{n, k} \overline{\hat{g}_n} \ \hat{h}_{n-k} \hat{h}_k \ \pare{n-k}^3 k \pare{\sgn\pare{n-k}\ \sgn k +1}. 
\end{equation}
As it was done in the proof of \eqref{lem:commutation} we observe that the non-zero contributions arise in the set 
$$ 
0 < k < n \quad \text{ or } \quad n < k < 0. 
$$ 
Thus, taking advantage of this symmetry,
\begin{equation*}
\cM \pare{g, h, h} = 4\nu \sum_{0<k<n}  \pare{n-k}^3 k \  \textnormal{Re}  \pare{\overline{\hat{g}_n} \ \hat{h}_{n-k} \hat{h}_k} , 
\end{equation*}
which entails the inequality
\begin{equation*}
\av{\cM \pare{g, h, h} } \leqslant 4 \sum_{0 < k < n} \pare{n-k}^3 k \ \av{\hat{g}_n} \av{\hat{h}_{n-k}} \av{\hat{h}_k}. 
\end{equation*}
Since the summation set is localized in the Fourier modes $ 0<k<n $ we can deduce the inequality 
$$ 
\pare{n-k}^3 k < n^2 \ \pare{n-k} k,
$$
thus, we can estimate as follows:
\begin{equation*}
\begin{aligned}
\av{\cM \pare{g, h, h} } & \leqslant C \norm{g}_{H^2}\norm{h}_{\dot{W}^{1, 4}}^2, \\
& \leqslant C \norm{g}_{H^2}\norm{h}_{H^{5/4}}^2, \\
& \leqslant C \norm{g}_{H^2}\norm{h}_{H^{1}}\norm{h}_{H^{3/2}}, \\
& \leqslant C \norm{g}_{H^2}\norm{h}_{H^{1/2}}^{1/2}\norm{h}_{H^{3/2}}^{3/2}. 
\end{aligned}
\end{equation*}
Equipped with \eqref{lem:commutation_cM}, we can hence conclude the proof of the $L^2$ estimates. In fact using Young inequality we deduce that
\begin{equation*}
\av{\cM\pare{f, f, f}} \leqslant \alpha \ \norm{f}_{H^{1/2}}^2 + \frac{C\nu^{4/3} }{\alpha}\norm{f}_{H^2}^{4/3}\norm{f}_{H^{3/2}}^2.
\end{equation*}
Using \eqref{lem:commutation} and the continuous embedding of $ H^2 $ into $ H^{3/2} $ we obtain that
\begin{equation*}
\begin{aligned}
\av{\cN\pare{f, f, f}} & \leqslant C \norm{f}_{H^{3/2}} \norm{f}_{H^{1/2}}^2, \\
& \leqslant C \norm{f}_{H^{2}} \norm{f}_{H^{1/2}}^2.
\end{aligned}
\end{equation*}
From here, taking $ \alpha =1/4 $ and if 
\begin{equation*}
\begin{aligned}
\norm{f\pare{t}}_{H^2} \leqslant \frac{1}{C} \ \min \set{1, \nu^{-\frac{1}{4}}} && \textnormal{for all } t \in \bra{0, T}, 
\end{aligned}
\end{equation*}
we deduce that
\begin{equation*}
\av{\cN\pare{f, f, f}} + \av{\cM \pare{f, f, f}} \leqslant \frac{\nu}{2}\norm{f}_{H^{3/2}}^2 + \frac{1}{2}  \norm{f}_{H^{1/2}}^2, 
\end{equation*}
lead us to
\begin{equation}\label{lem:commutation_nu>0}
\ddt \norm{f \pare{t} }_{L^2}^2 + \int_0^t \pare{\Big. \nu \norm{\Lambda^{3/2} f\pare{s}}_{L^2}^2 +  \norm{\Lambda^{1/2} f\pare{s}}_{L^2}^2} \dt  \leqslant \norm{f_0}_{L^2}^2. 
\end{equation}

\subsubsection*{ Step 1.2: $ H^2 $ estimates.}
We want now to prove global $ H^2 $ estimates for \eqref{eq:AD} stemming from small initial data. We apply the dyadic truncation $ \tqS $ to the left of \eqref{eq:AD}, we multiply the resulting equation for $ \tqS f $ and integrate in $ x $ obtaining
\begin{equation}\label{eq:H2_dyadic_estimate}
\frac{1}{2} \ddt \norm{\tqS f}_{L^2}^2 +\nu \norm{\tqS \Lambda^{3/2} f }_{L^2}^2 + \norm{\tqS \Lambda^{1/2} f }_{L^2}^2 = A_q-B_q+C_q+D_q,
\end{equation}
where
$$
A_q=\int \tqS \pare{f \ \Lambda f} \tqS \Lambda f \dx,
$$
$$
B_q=\int\tqS \pare{f \ \partial_x f }\tqS \partial_x f \ \dx,
$$
$$
C_q=\nu \int \tqS \pare{f \ \Lambda^3 f} \tqS \Lambda f \dx,
$$
$$
D_q=\int\tqS \pare{f \ \partial_x^3 f }\tqS \partial_x f \ \dx.
$$
We can use the estimates performed before (see in particular the r.h.s. of \eqref{eq:Hs_est2}) in order to deduce
\begin{equation}\label{eq:H2_nonlinear_estimate_1}
\begin{aligned}
\av{A_q - B_q} & \leqslant C b_q 2^{-4q} \ \norm{f}_{H^{2}} \norm{\Lambda^{1/2} f}_{H^2}^2. 
\end{aligned}
\end{equation}
We can hence now focus on the purely nonlinear part which is characteristic of \eqref{eq:AD} when $ \nu > 0 $, i.e. $ -\nu \pare{C_q + D_q} $. We can use Bony decomposition \ref{bony decomposition asymmetric} in order to decompose $ C_q $ and $ D_q $ as
\begin{equation*}
\begin{aligned}
C_q & = T^C_{1, q}+ T^C_{2, q} + T^C_{3, q} + R^C_q, \\
D_q & = T^D_{1, q}+ T^D_{2, q} + T^D_{3, q} + R^D_q,
\end{aligned}
\end{equation*}
where
\begin{align*}
T^C_{1, q} & = \int S_{q-1} f \  \tqS \Lambda^3 f \ \tqS \Lambda f \ \dx 
&
 T^D_{1, q} & = \int S_{q-1} f \  \tqS \partial_x^3 f \ \tqS \partial_x f \ \dx \\
T^C_{2, q} & = \sumf \int \bra{\tqS, \SQ f}\triangle_{q'} \Lambda^3 f \  \tqS \Lambda f \ \dx, & 
T^D_{2, q} & = \sumf \int \bra{\tqS, \SQ f}\triangle_{q'} \partial_x^3 f \ \tqS \partial_x f \ \dx, \\
T^C_{3, q} & = \sumf \int \pare{\SQ - \Sq}f \ \tqS \triangle_{q'} \Lambda^3 f \ \tqS \Lambda f \ \dx, &
T^D_{3, q} & = \sumf \int \pare{\SQ - \Sq}f \ \tqS \triangle_{q'} \partial_x^3 f \ \tqS \partial_x f\ \dx,\\
R^C_q & = \sumi \int \tqS \pare{\tQS f \ S_{q'+2}\Lambda^3 f} \ \tqS \Lambda f \dx, &
R^D_q & = \sumi \int \tqS \pare{\tQS f \ S_{q'+2}\partial_x^3 f} \ \tqS \partial_x f \dx. 
\end{align*}

We remark now that
\begin{equation*}
-\nu \pare{T^C_{1, q} + T^D_{1, q}} = \cM \pare{S_{q-1} f, \tqS f. \tqS f}.
\end{equation*}
We can apply the estimate \eqref{lem:commutation_cM} and \eqref{regularity_dyadic} in order to deduce the bound
\begin{equation*}
\begin{aligned}
\av{\cM \pare{S_{q-1} f, \tqS f. \tqS f} } & \leqslant C \nu \norm{f}_{H^2} \norm{\tqS \Lambda^{1/2} f}_{L^2}^{1/2} \norm{\tqS \Lambda^{3/2} f}_{L^2}^{3/2}, \\
& \leqslant Cb_q 2^{-4q} \nu \norm{f}_{H^2} \norm{ \Lambda^{1/2} f}_{H^2}^{1/2} \norm{ \Lambda^{3/2} f}_{H^2}^{3/2}. 
\end{aligned}
\end{equation*}

Next we focus on the remainder terms. Using H\"older inequality, Sobolev embeddings, \eqref{regularity_dyadic} and interpolation of Sobolev spaces it is possible to deduce the following estimate
\begin{equation*}
\begin{aligned}
\av{R^C_q} & \leqslant C \sumi \norm{S_{q'+2}\Lambda^3 f}_{L^2} \norm{\triangle_{q'-1} f}_{L^4} \norm{\tqS \Lambda f}_{L^4}, \\
& \leqslant C\sumi \norm{S_{q'+2}\Lambda^2 f}_{L^2} \norm{\triangle_{q'-1} \Lambda f}_{L^4} \norm{\tqS \Lambda f}_{L^4}, \\
& \leqslant C c_q 2^{-4q} \norm{f}_{H^2}\norm{\Lambda^{5/4} f}_{H^{2}}^2 \sumi 2^{2\pare{q-q'}}c_{q'}, \\
& \leqslant C b_q 2^{-4q} \norm{f}_{H^2}\norm{\Lambda^{5/4} f}_{H^{2}}^2, \\
& \leqslant Cb_q 2^{-4q}  \norm{f}_{H^2} \norm{ \Lambda^{1/2} f}_{H^2}^{1/2} \norm{ \Lambda^{3/2} f}_{H^2}^{3/2}. 
\end{aligned}
\end{equation*}
In the above estimates we have that
\begin{equation*}
b_q = c_q \pare{ \pare{1_{p<4} 2^{2q'}} \star_{q'} c_{q'}}_q, 
\end{equation*}
which is $ \ell^2 $ as long as $ \pare{c_q}_q\in \ell^1 $. Similar computations holds for the term $ R^D_q $, from where we deduce that
\begin{equation*}
\av{R^D_q}\leqslant Cb_q 2^{-4q}  \norm{f}_{H^2} \norm{ \Lambda^{1/2} f}_{H^2}^{1/2} \norm{ \Lambda^{3/2} f}_{H^2}^{3/2}. 
\end{equation*}

Next we study the term $ T^C_{2, q} $. Using Lemma \ref{estimates commutator} we obtain
\begin{equation*}
\begin{aligned}
\av{T^C_{2, q}} & \leqslant \sumf 2^{-q}\norm{\Lambda f}_{L^\infty} \norm{\triangle_{q'} \Lambda^3 f}_{L^2} \norm{\tqS \Lambda f}_{L^2}. 
\end{aligned}
\end{equation*}
Now, since $ \norm{\triangle_{q'} \Lambda^3 f}_{L^2}\lesssim 2^{\frac{3q'}{2}} \norm{\triangle_{q'} \Lambda^{3/2} f}_{L^2} $ and as a consequence of $ 2^{q'}\sim 2^q $ for $ \av{q-q'}\leqslant 4 $ we find that
\begin{equation*}
\begin{aligned}
\av{T^C_{2, q}} & \leqslant C \sumf \norm{f }_{H^2} \norm{\triangle_{q'}\Lambda^{3/2}f}_{L^2}\norm{\tqS\Lambda^{3/2}f}_{L^2}, \\
& \leqslant C b_q 2^{-4q} \norm{f}_{H^2}\norm{\Lambda^{3/2}f}_{H^2}^2. 
\end{aligned}
\end{equation*}
We can apply the very same computations to $ T^D_{2, q} $, which give
\begin{equation*}
\begin{aligned}
\av{T^D_{2, q}} & \leqslant C b_q 2^{-4q} \norm{f}_{H^2}\norm{\Lambda^{3/2}f}_{H^2}^2. 
\end{aligned}
\end{equation*}
The terms $ T^C_{3, q}, \ T^D_{3, q} $ enjoy analog bounds ans are the overall more regular terms, being composed by elements localized in dyadic annuli, thus
\begin{equation}\label{eq:H2_nonlinear_estimate_2}
\nu \av{C_q + D_q} \leqslant C \nu b_q 2^{-4q} \bra{ \norm{f}_{H^2} \norm{ \Lambda^{1/2} f}_{H^2}^{1/2} \norm{ \Lambda^{3/2} f}_{H^2}^{3/2} + \norm{f}_{H^2}\norm{\Lambda^{3/2}f}_{H^2}^2 }. 
\end{equation}

Collecting \eqref{eq:H2_nonlinear_estimate_1} and \eqref{eq:H2_nonlinear_estimate_2} in \eqref{eq:H2_dyadic_estimate}, multiplying the resulting inequality for $ 2^{4q} $, summing up in $ q\in \bZ $ and applying Young convexity inequality we conclude the following inequality

\begin{equation*}
\frac{1}{2} \ddt \norm{f}_{H^2}^2 + \nu \norm{\Lambda^{3/2} f }_{H^2}^2 + \norm{\Lambda^{1/2} f }_{H^2}^2 \leqslant \pare{\frac{1}{4} + C\norm{f}_{H^2}}\norm{\Lambda^{1/2} f }_{H^2}^2 
+ C \pare{\nu \norm{f}_{H^2} + \nu ^{\frac{4}{3}}  \norm{f}_{H^2}^{\frac{4}{3}}  } \norm{\Lambda^{3/2} f }_{H^2}^2 .
\end{equation*}
From the previous inequality, if the initial data satisfies
\begin{equation}\label{eq:smallness_condition_H2}
\norm{f_0}_{H^2\pare{\bT}} \leqslant \frac{1}{C} \ \min \set{1, \ \nu^{-\frac{1}{4}}}, 
\end{equation}
we obtain the desired global bound
\begin{equation}\label{eq:H2_energy_estimate}
\norm{f\pare{t}}_{H^2\pare{\bT}}^2 + \int_0^t \bra{ \nu \norm{ \Lambda^{3/2} f\pare{s}}_{H^2\pare{\bT}}^2 + \norm{ \Lambda^{1/2} f\pare{s}}_{H^2\pare{\bT}}^2} \dt \leqslant \norm{f_0}_{H^2\pare{\bT}}^2. 
\end{equation}

\subsubsection*{Step 2: Approximated solutions and passing to the limit} The proof is again a standard approximation argument. Using the projection operator defined in \eqref{truncation}, we consider the approximated problems
\begin{equation}
\left\lbrace
\begin{aligned}
& \partial_t f_n + \nu \Lambda^3 f_n + \Lambda f_n =  -\cJ_n\bra{ \nu \bra{ \Big. \Lambda\pare{f_n\ \Lambda^3 f_n} - \partial_x \pare{f_n\ \partial_x^3 f_n}} + \bra{ \Big. \Lambda\pare{f_n\ \Lambda f_n} + \partial_x \pare{f_n\ \partial_x f_n}}}, \\
& \left. f_n \right|_{t=0} = \cJ_n f_0. 
\end{aligned}
\right. 
\end{equation}
We can now define the space
\begin{equation*}
\widetilde{\cH}_n = \set{u\in \cD' \ \left| \ u\in H^2, \ \textnormal{supp} \hat{u} \subset B_n\pare{0} \Big. \right. }. 
\end{equation*}
Using Cauchy-Lipschitz theorem and the estimate \eqref{eq:H2_energy_estimate} we deduce that if $ f_0 $ satisfies the smallness hypothesis \eqref{eq:smallness_condition_H2} then $ f_n\in C^1 \pare{\bR_+; \widetilde{\cH}_n} \cap L^2\pare{\bR_+; H^{\frac{7}{2}}} $. Moreover $ \pare{\partial_t f_n}_n $ is uniformly bounded in $ L^2\pare{\bR_+; H^{-N}} $ for $ N $ sufficiently large, and invoking Aubin-Lions lemma we deduce that
\begin{equation*}
\pare{f_n}_n \ \text{ is compact in } L^2\pare{\bra{0, T}; H^{\frac{7}{2}-\epsilon} }, \ \forall \ \epsilon, T >0. 
\end{equation*} 
Passing to a subsequence if necessary we obtain the convergence $ f_n\to f $ in  $ L^2\pare{\bra{0, T}; H^{\frac{7}{2}-\epsilon} } $ and the limit element $ f $ satisfies the energy inequality \eqref{eq:H2_energy_estimate}, so
\begin{equation*}
f\in C\pare{\bR_+; H^2}\cap L^2\pare{\bR_+; H^{\frac{7}{2}}}. 
\end{equation*}

\section*{Acknowledgments}
{The research of S.S. is supported by the Basque Government through the BERC 2018-2021 program and by Spanish Ministry of Economy and Competitiveness MINECO through BCAM Severo Ochoa excellence accreditation SEV-2017-0718 and through project MTM2017-82184-R funded by (AEI/FEDER, UE) and acronym "DESFLU".}

\appendix

\section{Elements of Littlewood-Paley theory.}\label{elements LP}

A tool that will be widely used all along the paper is the theory of Littlewood--Paley, which consists in doing a dyadic cut-off of the  frequencies. Let us define the (homogeneous)  truncation operators as follows:
\begin{align*}
\tv u= & \sum_{n\in\mathbb{Z}^3} \hat{u}_n \varphi \left(\frac{\left|\check{n}\right|}{2^q}\right) e^{i\check{n} \cdot x}, &\text{for }& q\in\bZ,
\end{align*}
where $u\in\mathcal{D}'\left(\mathbb{T}^3 \right)$ and  $\hat{u}_n$ are the Fourier coefficients of $u$. The function $\varphi$ is a  smooth function with compact support such that
\begin{align*}
 \text{supp}\;\varphi \subset& \;\mathcal{C}\left( \frac{3}{4},\frac{8}{3}\right),
\end{align*}
and such that for all $t\in\mathbb{R}$,
$$
\sum_{q\in\bZ} \varphi \left( 2^{-q}t\right)=1.
$$

Let us define further the low frequencies cut-off operator
$$
S_q u= \sum_{q'\leqslant q-1}\Tv u.
$$

The dyadic decomposition turns out to be very useful also when it comes to study the product between two distributions. We can in fact, at least formally, write for two distributions $u$ and $v$
\begin{align}\label{decomposition vertical frequencies}
u=&\sum_{q\in\mathbb{Z}}\tv u ; &
v=&\sum_{q'\in\mathbb{Z}}\Tv v;&
u\cdot v = & \sum_{\substack{q\in\mathbb{Z} \\ q'\in\mathbb{Z}}}\tv u \cdot \Tv v.
\end{align}

Paradifferential calculus is  a mathematical tool for splitting the above sum in three parts
$$
u\cdot v = T_u v+ T_v u + R\left(u,v\right),
$$
where
\begin{align*}
T_u v=& \sum_q S_{q-1} u\; \tqS v, &
 T_v u= & \sum_{q'} S_{q'-1} v \; \Tv u,&
 R\left( u,v \right) = & \sum_k \sum_{\left| \nu\right| \leqslant 1} \triangle_k  u\; \triangle_{k+\nu} v.
\end{align*}
The following almost orthogonality properties hold
\begin{align*}
\tqS \left( \Sq a \Tv b\right)=&0, & \text{if }& \left|q-q'\right|\geqslant 5,\\
\tqS \left( \Tv a \triangle_{q'+\nu}b\right)=&0, & \text{if }& q'< q-4,\; \left| \nu \right|\leqslant 1,
\end{align*}
and hence we will often use the following relation
\begin{equation*}
\begin{aligned}
\tqS\left( u\cdot v \right)= &\sum_{\left| q -q'\right| \leqslant 4} \tqS\left(S_{q'-1} v\; \Tv u\right) +
\sum_{\left| q -q'\right| \leqslant 4} \tqS\left(S_{q'-1} u\; \Tv v\right)+
\sum_{q'\geqslant q-4}\sum_{|\nu|\leqslant 1}\tqS\left(  \Tv a \triangle_{q'+\nu}b\right) ,\\
=& \sum_{\left| q -q'\right| \leqslant 4} \tqS\left(S_{q'-1} v \; \Tv u\right) + \sum_{q'>q-4} \tqS\left( S_{q'+2} u \Tv v\right).
\end{aligned}
\end{equation*}

In the paper \cite{chemin_lerner} J.-Y. Chemin and N. Lerner introduced the following decomposition which will  be very useful in our context
\begin{equation}
\label{bony decomposition asymmetric}
\tqS \left(uv\right) = S_{q-1} u\; \tqS v +\sum_{|q-q'|\leqslant 4} \left\lbrace\left[ \tqS, \SQ u \right]\Tv v + \left( \SQ u-\Sq u \right) \tqS\Tv v\right\rbrace
\\
 + \sum_{q'>q-4} \tqS \left(  S_{q'+2} v\;\Tv u\right),
\end{equation}
where the commutator $\left[\tqS, a\right]b$ is defined as
$$
\left[\tqS, a\right]b= \tqS \left( ab \right) - a \tqS b.
$$
There is an interesting relation of regularity between dyadic blocks and full function in the Sobolev spaces, i.e.
\begin{equation}
\label{regularity_dyadic}
\left\| \tqS f \right\|_{L^p\pare{\bT}} \leqslant C c_q^{\pare{p}} 2^{-qs}\left\| f \right\|_{W^{s, p}\pare{\bT}},
\end{equation}
with $ \left\| \left\lbrace c_q^{\pare{p}} \right\rbrace_{q\in\mathbb{Z}} \right\|_{\ell^p\left( \mathbb{Z} \right)}\equiv 1 $, if $ p=2 $ we denote $ \set{c_q^{\pare{2}}}_q = \set{c_q}_q $ for simplicity. In the same way we denote as $ b_q $ a sequence in $ \ell^1 \left( \mathbb{Z} \right) $ such that $ \sum_q \left| b_q \right| \leqslant 1$.\\

Finally we state a lemma that shows that the commutator with the dyadic block in the vertical frequencies is a regularizing operator. The proof of such lemma can be found in \cite{paicu_NS_periodic}.
\begin{lemma}\label{estimates commutator}
Let $\mathbb{T}^d$ be a $ d $-dimensional torus and $p,r,s$ real positive numbers such that $ p,r,s\in\bra{1, \infty} $ and $\frac{1}{p}=\frac{1}{r}+\frac{1}{s}$. There exists a constant $C$ such that for all vector fields $u$ and $v$ on $\mathbb{T}^d$ we have the inequality
$$
\left\| \left[ {\tqS} , u\right] v \right\|_{L^p}\leqslant C  2^{-q}\left\|\nabla u \right\|_{L^r} \left\|v\right\|_{L^s}.
$$
\end{lemma}

\begin{footnotesize}
\bibliographystyle{plain}

\end{footnotesize}

\end{document}